\documentclass[12pt,reqno]{amsart}

\usepackage{amssymb,hyperref,xcolor}

\usepackage[all]{xy}

\newtheorem{theorem}{Theorem}[section]
\newtheorem{proposition}[theorem]{Proposition}
\newtheorem{lemma}[theorem]{Lemma}
\newtheorem{corollary}[theorem]{Corollary}

\theoremstyle{definition}
\newtheorem{definition}[theorem]{Definition}
\newtheorem{remark}[theorem]{Remark}
\newtheorem{assumption}[theorem]{Assumption}

\numberwithin{equation}{section}

\begin{document}

\baselineskip=15pt

\title[Deformations of parabolic connections and parabolic opers]{Infinitesimal deformations
of parabolic connections and parabolic opers}

\author[I. Biswas]{Indranil Biswas}

\address{Department of Mathematics, Shiv Nadar University, NH91, Tehsil
Dadri, Greater Noida, Uttar Pradesh 201314, India}

\email{indranil.biswas@snu.edu.in, indranil29@gmail.com}

\author[S. Dumitrescu]{Sorin Dumitrescu}

\address{Universit\'e C\^ote d'Azur, CNRS, LJAD, France}

\email{dumitres@unice.fr}

\author[S. Heller]{Sebastian Heller}

\address{Department of Mathematics and Institute of Mathematical Sciences, 
The Chinese University of Hong Kong, China}

\email{s.g.heller@icloud.com}

\author[C. Pauly]{Christian Pauly (corresponding author)}

\address{Universit\'e C\^ote d'Azur, CNRS, LJAD, France}

\email{pauly@unice.fr}

\subjclass[2010]{14H60, 33C80, 53C07.}

\keywords{Parabolic oper, logarithmic Atiyah bundle, isomonodromy, monodromy map}

\date{}

\begin{abstract} 
We compute the infinitesimal deformations of quadruples $(X,\, S,\, E_*,\, D)$, 
where $X$ is a compact Riemann surface with $n$ ordered marked points $S$ and $E_*$ is a parabolic 
vector bundle on $X$, with parabolic structure over the divisor $S$, equipped
with a parabolic connection $D$. Using this, we compute the infinitesimal deformations of $(X,\, S,\,
D)$, where $D$ is a parabolic ${\rm SL}(r,{\mathbb C})$--oper on $(X,\, S)$. It is shown that the monodromy 
map, from the moduli space of triples $(X,\, S,\, D)$, where $D$ is a parabolic ${\rm 
SL}(r,{\mathbb C})$--oper on $(X,\, S)$, to the ${\rm SL}(r,{\mathbb C})$--character 
variety of $X\setminus S$, is an immersion.
\end{abstract}

\maketitle

\tableofcontents

\section{Introduction}\label{se1}

Opers were introduced by Beilinson and Drinfeld \cite{BD1}, \cite{BD2}. These objects have
turned out to be very important in diverse topics, for example
in geometric Langlands correspondence, nonabelian Hodge theory, some branches of
mathematical physics, differential equations, et cetera; see \cite{BF}, \cite{DFKMMN},
\cite{FT}, \cite{FG}, \cite{FG2}, \cite{CS}, \cite{Fr}, \cite{Fr2}, \cite{KSZ}, \cite{MR},
\cite{BSY} and references therein.

Parabolic vector bundles were introduced by Mehta and Seshadri in \cite{MS}. In \cite{BDP}, 
${\rm SL}(r, {\mathbb C})$--opers in the setting of parabolic vector bundles were 
introduced, and their basic properties were established. The aim here is to further investigate ${\rm SL}(r, {\mathbb C})$--opers
in the context of parabolic vector bundles.

More precisely, we will generalize to parabolic opers a certain result proved in the interesting paper 
\cite[Theorem 1.3]{Sa}. The result in \cite{Sa} considers the
moduli space of $G$--opers over a compact Riemann surface which is in fact a holomorphic
fiber bundle that fibers over the Teichm\"uller space of the surface, and 
the result of \cite{Sa} states that the holonomy map from this moduli space
to the corresponding Betti moduli space is a holomorphic 
immersion for any complex simple Lie group $G$ of adjoint type. We note that this
result of \cite{Sa} generalizes the previous results known for
$G \,=\, {\rm PSL}(2, {\mathbb C})$ and complex projective structures.

For the aforementioned generalization, we will need to study the infinitesimal deformations
and the monodromy map of parabolic ${\rm SL}(r, {\mathbb C})$--opers. Computation of 
the infinitesimal deformations of parabolic opers entails computation of the infinitesimal 
deformations of quadruples of the form $(X,\, S,\, E_*,\, D)$, where $X$ is a compact 
Riemann surface equipped with $n$ ordered marked points $S$ and $E_*$ is a parabolic vector bundle on $X$,
with parabolic structure over the divisor $S$, while $D$ is a parabolic connection on
the parabolic vector bundle $E_*$.

We introduce a parabolic analog of the Atiyah bundle and the Atiyah exact sequence.
Given a parabolic vector bundle $E_*$ on $(X,\, S)$, its Atiyah bundle $\text{At}(E_*)$ fits in
the short exact sequence of holomorphic vector bundles
\begin{equation}\label{i1}
0\, \longrightarrow\, \text{End}^P(E_*)\, \longrightarrow\,
{\rm At}(E_*) \, \stackrel{\sigma}{\longrightarrow}\,
TX\otimes{\mathcal O}_X(-S) \, \longrightarrow\, 0,
\end{equation}
where $\text{End}^P(E_*)$ denotes the sheaf of quasi-parabolic flag preserving endomorphisms
of $E_*$ (see \eqref{e2} for the
quasi-parabolic flags), ${\rm At}(E_*)$ is a certain subsheaf of ${\rm Diff}^1(E,\, E)$ (see
Definition \ref{def1}) and $\sigma$ is the restriction of the symbol homomorphism
$$
{\rm Diff}^1(E,\, E) \ \longrightarrow\ \text{End}_{{\mathcal O}_X}(E)\otimes TX
$$
on the sheaf of holomorphic differential operators;
the sequence in \eqref{i1} is the Atiyah exact sequence in the setting of
parabolic vector bundles. We show the following:

\begin{proposition}[{Lemma \ref{lem1} and Corollary \ref{cor1}}] \label{propi}
\mbox{}
\begin{enumerate}
\item A holomorphic splitting of \eqref{i1} produces a logarithmic connection on
the holomorphic vector bundle $E$ underlying $E_*$ such that the residues preserve
the quasi-parabolic flags of $E_*$, and

\item a parabolic connection on $E_*$ is a holomorphic splitting of \eqref{i1} 
such that the eigenvalues of the residues are given by the parabolic weights
(see \eqref{e2a} for the parabolic weights).
\end{enumerate}
\end{proposition}

We prove the following (see Lemma \ref{lem2}):

\begin{lemma}\label{lemi1}
The infinitesimal deformations of the triple
$(X,\, S,\, E_*)$ are parametrized by the cohomology $H^1(X,\, {\rm At}(E_*))$
of the parabolic Atiyah bundle.
\end{lemma}

Now, let $D$ be a connection on the parabolic vector bundle $E_*$ over $(X,\, S)$. We
assume that the local monodromy of $D$ around each point of $S$ is semisimple (meaning
diagonalizable). Let
$$
\mathcal{D}_0\ :\ \text{End}^P(E_*)\ \longrightarrow\ \text{End}^n(E_*)\otimes K_X
$$
be the corresponding logarithmic connection on $\text{End}^P(E_*)$, where
$\text{End}^n(E_*)\, \subset\, \text{End}^P(E_*)$ is the 
sheaf of endomorphisms nilpotent with respect to the quasi-parabolic flags of $E_*$ and
$K_X$ denotes the holomorphic cotangent bundle of $X$.
We show that this operator $\mathcal{D}_0$ extends to a holomorphic differential operator
$$
\mathcal{D}\ :\ {\rm At}(E_*)\ \longrightarrow\ \text{End}^n(E_*)\otimes K_X
\otimes{\mathcal O}_X(S)
$$
(see \eqref{e21}). Let
${\mathcal B}_\bullet$ denote the following two-term complex of sheaves on $X$
$$
{\mathcal B}_\bullet\,\,:\,\, {\mathcal B}_0\,=\, {\rm At}(E_*)\,
\stackrel{\mathcal{D}}{\longrightarrow}\, {\mathcal B}_1\,=\,\text{End}^n(E_*)
\otimes K_X\otimes{\mathcal O}_X(S),
$$
where ${\mathcal B}_i$ is at the $i$-th position.

We prove the following (see Lemma \ref{lem4}):

\begin{lemma}\label{lemi2}
The infinitesimal deformations of the quadruple $(X,\, S,\, E_*,\, D)$
are parametrized by the first hypercohomology group
${\mathbb H}^1({\mathcal B}_\bullet)$.
\end{lemma}

Now assume that $D$ is a parabolic ${\rm SL}(r, {\mathbb C})$--oper (the definition
of a parabolic ${\rm SL}(r, {\mathbb C})$--oper is recalled in Section \ref{se4}). Since $D$
is a parabolic ${\rm SL}(r, {\mathbb C})$--oper, the local monodromy of $D$ around each point of
$S$ is semisimple (see Lemma \ref{os}). In \eqref{e41} and \eqref{g1} we construct the holomorphic
vector bundles $\text{At}_X(r)$ and $\text{ad}^n_1\left(\text{Sym}^{r-1}(\mathcal{E}_*)\right)$
respectively on $X$. Using $D$ we construct a differential operator
$$
{\mathcal D}_{\mathcal B}\ :\ \text{At}_X(r)\
\longrightarrow\ \text{ad}^n_1\left(\text{Sym}^{r-1}(\mathcal{E}_*)\right)\otimes
K_X\otimes{\mathcal O}_X(S)
$$
(see \eqref{e42}). Let ${\mathcal C}_\bullet$ denote the following two-term complex of
sheaves on $X$
$$
{\mathcal C}_\bullet\,\,:\,\, {\mathcal C}_0\,=\, \text{At}_X(r)\,
\stackrel{\mathcal{D}_{\mathcal B}}{\longrightarrow}\, {\mathcal C}_1\,=\,\text{ad}^n_1
\left(\text{Sym}^{r-1}(\mathcal{E}_*)\right)\otimes K_X\otimes{\mathcal O}_X(S)\,
$$
where ${\mathcal C}_i$ is at the $i$-th position.

We prove the following (see Theorem \ref{thm1}):

\begin{theorem}\label{thmi1}
The space of all infinitesimal deformations of the triple $(X,\, S,\, D)$, where $D$ is
a parabolic ${\rm SL}(r,{\mathbb C})$--oper on $(X,\, S)$, is given by the hypercohomology
group ${\mathbb H}^1({\mathcal C}_\bullet)$.
\end{theorem}

A reformulation of the above result is proved in Corollary \ref{cor2}.

In Theorem \ref{thm2} we prove the following:

\begin{theorem}\label{thmi2}
The monodromy map from the moduli space of triples $(X,\, S,\, D)$, where $D$ is
a parabolic ${\rm SL}(r,{\mathbb C})$--oper on $(X,\, S)$, to the ${\rm SL}(r,
{\mathbb C})$--character variety for $X\setminus S$ is an immersion.
\end{theorem}

The isomonodromy condition defines a holomorphic foliation on the
moduli space of quadruples $(X,\, S,\, E_*,\, D)$, where $D$ is a connection
on the parabolic vector bundle $E_*$ over $(X,\, S)$.
The proof of Theorem \ref{thm2} involves computing this foliation. This
computation is carried out in Lemma \ref{lem5}.

\medskip

In the appendix we give an alternative definition of a parabolic ${\rm SL}(r,{\mathbb C})$--oper in terms 
of $\mathbb{R}$-filtered sheaves as introduced and studied by Maruyama and Yokogawa. This definition is 
conceptually closer to the definition of an ordinary ${\rm SL}(r,{\mathbb C})$--oper and clarifies the one 
given in \cite{BDP}.

\section{Holomorphic connections and the Atiyah bundle}

\subsection{Atiyah bundle for parabolic vector bundles}

Let $X$ be a compact connected Riemann surface. Fix a finite subset of $n$
ordered distinct points
\begin{equation}\label{e1}
S\, :=\, \{x_1,\, \cdots,\, x_n\}\, \subset\, X.
\end{equation}
The reduced effective divisor $x_1+\ldots + x_n$ on $X$ will also be denoted
by $S$.

A \textit{quasi-parabolic structure} on a holomorphic vector bundle $E$ on $X$
is a filtration of subspaces of the fiber $E_{x_i}$ of $E$ over $x_i$
\begin{equation}\label{e2}
E_{x_i}\,=\,E_{i,1}\,\supset\, E_{i,2}\,\supset\, \cdots\,
\supset \,E_{i,l_i} \,\supset\, E_{i,l_i+1}\,=\, 0
\end{equation}
for every $1\, \leq\, i\, \leq\, n$. A \textit{parabolic structure} on
$E$ is a quasi-parabolic structure as above together with a string of real numbers
\begin{equation}\label{e2a}
0\,\leq\, \alpha_{i,1} \,< \,\alpha_{i,2} \,<\,
\cdots \,<\, \alpha_{i,l_i}\,<\, 1
\end{equation}
for every $1\, \leq\, i\, \leq\, n$. The above number $\alpha_{i,j}$
is called the parabolic weight of the subspace $E_{i,j}$ in \eqref{e2}. The divisor $S$
is known as the parabolic divisor. (See \cite{MS}, \cite{MY}, \cite{Bh}.)

A parabolic vector bundle is a holomorphic vector bundle $E$ equipped with a parabolic structure
$(\{E_{i,j}\},\, \{\alpha_{i,j}\})$. For notational convenience,
$(E,\, (\{E_{i,j}\},\, \{\alpha_{i,j}\}))$ will be denoted by $E_*$.

\begin{assumption}\label{asm1}
Throughout we will work with the assumption that all the parabolic weights $\alpha_{i,j}$ are rational numbers.
This assumption is needed to establish an equivalence of categories between parabolic
vector bundles over $X$ and Galois-equivariant vector bundles
over some Galois ramified cover $Y$ of $X$ (see \cite{Bi0}). We will use this equivalence of categories in the proof of Lemma 3.1.
We assume that each $x_i$ has at least one nonzero parabolic weight.
We also assume that $2({\rm genus}(X)-1)+n \, >\, 0$.
\end{assumption}

For any $1\, \leq\, i\, \leq\, n$ and $1\, \leq\, j \,\leq\, l_i+1$, let
${\mathcal E}_{i,j}\, \longrightarrow\, X$ be the holomorphic vector
bundle defined by the
following short exact sequence of coherent analytic sheaves on $X$:
\begin{equation}\label{e3}
0\, \longrightarrow\, {\mathcal E}_{i,j}\, \longrightarrow\, E \, \longrightarrow\,
E_{x_i}/E_{i,j} \, \longrightarrow\, 0
\end{equation}
(see \eqref{e2}). Since $E_{x_i}/E_{i,j}$ is supported on the point $x_i$, the subsheaf
${\mathcal E}_{i,j}$ of
$E$ actually coincides with $E$ over the open subset $X\setminus \{x_i\}\, \subset\, X$.

The space of all holomorphic sections, over an open subset $U\, \subset\, X$, of
a holomorphic vector bundle $V\, \longrightarrow\, X$ is denoted by $\Gamma(U,\, V)$.

Let
\begin{equation}\label{e4}
\text{End}^P(E_*)\, \subset\, \text{End}(E)
\end{equation}
be the coherent analytic subsheaf such that for any open subset
$U\, \subset\, X$, the subspace
$$
\Gamma\left(U,\, \text{End}^P(E_*)\right)\, \subset\, \Gamma(U,\, \text{End}(E))
$$
consists of all ${\mathcal O}_U$--linear homomorphisms
$s\, :\, E\big\vert_U \, \longrightarrow\, E\big\vert_U$ satisfying the condition that
$$
s({\mathcal E}_{i,j}\big\vert_U)\, \subset\, {\mathcal E}_{i,j}\big\vert_U
$$
for all $1\, \leq\, i \, \leq\, n$ with $x_i\, \in\, U$ and all $1\, \leq\, j \,\leq\, l_i$, where
${\mathcal E}_{i,j}$ is defined in \eqref{e3}.

Let
\begin{equation}\label{e5}
\text{End}^n(E_*)\, \subset\, \text{End}^P(E_*)
\end{equation}
be the coherent analytic subsheaf consisting of all ${\mathcal O}_U$--linear homomorphisms
$s\, :\, E\big\vert_U \, \longrightarrow\, E\big\vert_U$ satisfying the condition that
for any open subset $U\, \subset\, X$,
$$
s({\mathcal E}_{i,j}\big\vert_U)\, \subset\, {\mathcal E}_{i,j+1}\big\vert_U
$$
for all $1\, \leq\, i \, \leq\, n$ with $x_i\, \in\, U$ and all $1\, \leq\, j \,\leq\, l_i$.

\begin{remark}\label{rem1}
It is customary to define $\text{End}^P(E_*)$ as the subsheaf of $\text{End}(E)$ that
preserves the subspace $E_{i,j}\, \subset\,E_{x_i}$ for
all $x_i\, \in\, U$ and all $1\, \leq\, j \,\leq\, l_i$. Similarly,
$\text{End}^n(E_*)$ is defined to be the subsheaf of $\text{End}^P(E_*)$ that
takes any $E_{i,j}$ to $E_{i,j+1}$. While the definitions in
\eqref{e4} and \eqref{e5} are equivalent to these, we will see that the definitions
in \eqref{e4} and \eqref{e5} are more useful for our purpose.
\end{remark}

Using the pairing $(\text{End}(E)\otimes {\mathcal O}_X(S))^{\otimes 2}\,\longrightarrow\,
{\mathcal O}_X(2S)$ defined by trace $$A\otimes B\, \longmapsto\, \text{trace}(AB)\, ,$$ we have
$$
\text{End}^P(E_*)^*\,=\, \text{End}^n(E_*)\otimes {\mathcal O}_X(S).
$$

For any integer $k\, \geq\, 0$, let ${\rm Diff}^k(E,\, E)$ be the holomorphic vector 
bundle on $X$ given by the sheaf of all holomorphic differential operators, of order at 
most $k$, from $E$ to itself. We have the following short exact sequence of holomorphic 
vector bundles on $X$
\begin{equation}\label{e6}
0\, \longrightarrow\, {\rm Diff}^0(E,\, E)\,=\, \text{End}_{{\mathcal O}_X}(E)\,
\longrightarrow\, {\rm Diff}^1(E,\, E)
\end{equation}
$$
\stackrel{\sigma_0}{\longrightarrow}\, \text{End}_{{\mathcal O}_X}(E)
\otimes TX\,=\, E\otimes E^*\otimes TX \, \longrightarrow\, 0\, ,
$$
where $TX$ is the holomorphic tangent bundle of $X$ and $\sigma_0$ is the symbol map
for first-order differential operators. Let
\begin{equation}\label{e7}
{\rm Diff}^1_P(E,\, E)\, \subset\, {\rm Diff}^1(E,\, E)
\end{equation}
be the coherent analytic subsheaf consisting of all differential operators
$D_U\, :\, E\big\vert_U \, \longrightarrow\, E\big\vert_U$,
where $U\, \subset\, X$ is any open subset, satisfying the condition that
for any $s\, \in\, \Gamma(U,\, {\mathcal E}_{i,j})$,
$$
D_U(s) \, \in\, \Gamma(U,\, {\mathcal E}_{i,j})
$$
for all $x_i\, \in\, U$ and all $1\, \leq\, j \,\leq\, l_i$, where ${\mathcal E}_{i,j}$
is constructed in \eqref{e3}.

Comparing the definitions given in \eqref{e7} and \eqref{e4} we conclude that
\begin{equation}\label{e11}
{\rm Diff}^1_P(E,\, E)\bigcap \text{End}_{{\mathcal O}_X}(E)\,=\, \text{End}^P(E_*)\, ;
\end{equation}
the above intersection takes place inside ${\rm Diff}^1(E,\, E)$ (see \eqref{e6}
and \eqref{e7}). Consider the subsheaf
\begin{equation}\label{e13}
TX\otimes{\mathcal O}_X(-S)\, \subset\, TX\,=\, {\rm Id}_E\otimes TX
\, \subset\, \text{End}(E)\otimes TX ,
\end{equation}
where $S$ is the divisor in \eqref{e1}. Define
\begin{equation}\label{e9}
{\rm At}(E_*)\, :=\, {\rm Diff}^1_P(E,\, E)\bigcap \sigma^{-1}_0(TX\otimes{\mathcal O}_X(-S))
\, \subset\, {\rm Diff}^1(E,\, E),
\end{equation}
where $\sigma_0$ is the projection in \eqref{e6} and ${\rm Diff}^1_P(E,\, E)\, \subset\,
{\rm Diff}^1(E,\, E)$ is the subsheaf in \eqref{e7}; in \eqref{e9},
$TX\otimes{\mathcal O}_X(-S)$ is considered as a subsheaf of $\text{End}(E)\otimes TX$
by sending any $s\,\in\, \Gamma(U,\, TX\otimes{\mathcal O}_X(-S))$ to $({\rm Id}_E\big\vert_U)\otimes
s$. Now from \eqref{e6} and \eqref{e11}
we get the following short exact sequence of holomorphic vector bundles on $X$
\begin{equation}\label{e10}
0\, \longrightarrow\, \text{End}^P(E_*)\, \longrightarrow\,
{\rm At}(E_*) \, \stackrel{\sigma}{\longrightarrow}\,
TX\otimes{\mathcal O}_X(-S) \, \longrightarrow\, 0 ,
\end{equation}
where $\sigma$ is the restriction of $\sigma_0$ to ${\rm At}(E_*)\, \subset\,
{\rm Diff}^1(E,\, E)$. It is straightforward to check that the homomorphism
$\sigma$ in \eqref{e10} is surjective; indeed, this follows immediately from
the fact that $TX\otimes{\mathcal O}_X(-S) \,\subset\, \sigma_0({\rm Diff}^1_P(E,\, E))$.

\begin{definition}\label{def1}
The vector bundle ${\rm At}(E_*)$ in \eqref{e9} will be called the \textit{Atiyah
bundle} for the parabolic vector bundle $E_*$, and the sequence in \eqref{e10} will be called
the \textit{Atiyah exact sequence} for the parabolic vector bundle $E_*$.
\end{definition}

When $S$ is the zero divisor (meaning $n\,=\, 0$ in \eqref{e1}), then ${\rm At}(E_*)$ is the
usual Atiyah bundle $\text{At}(E)$ for $E$, and \eqref{e10} is the usual Atiyah
exact sequence for $E$. (See \cite{At}.)

\subsection{Holomorphic connections on a parabolic vector bundle}

As before, the holomorphic cotangent bundle of $X$ will be denoted by $K_X$.

Let $V$ be a holomorphic
vector bundle on $X$. A \textit{logarithmic connection} on $V$ singular along
$S$ is a holomorphic differential operator of order one
$$
D\, :\, V\, \longrightarrow\, V\otimes K_X\otimes {\mathcal O}_X(S)
$$
satisfying the Leibniz identity, which says that
\begin{equation}\label{e12}
D(fs)\,=\, fD(s)+ s\otimes df
\end{equation}
for any locally
defined holomorphic function $f$ on $X$ and any locally defined holomorphic section $s$ of $V$.

We note that any logarithmic connection on $V$ is flat because $\Omega^{2,0}_X\,=\,0$.

Take a point $x_i\, \in\, S$. The fiber of
$K_X\otimes{\mathcal O}_X(S)$ over $x_i$ is identified with $\mathbb C$ by the Poincar\'e
adjunction formula \cite[p.~146]{GH}. To explain this isomorphism
\begin{equation}\label{pa}
(K_X\otimes{\mathcal O}_X(S))_{x_i} \, \stackrel{\sim}{\longrightarrow}\, {\mathbb C}\, ,
\end{equation}
let $z$ be a holomorphic coordinate function on $X$ defined on an analytic
open neighborhood of $x_i$ such that $z(x_i)\,=\, 0$. Then we have the isomorphism
${\mathbb C}\, \longrightarrow\, (K_X\otimes{\mathcal O}_X(S))_{x_i}$ that sends any
$c\, \in\, \mathbb C$
to $c\cdot \frac{dz}{z}(x_i)\,\in\, (K_X\otimes{\mathcal O}_X(S))_{x_i}$. It is straightforward
to check that this map ${\mathbb C}\, \longrightarrow\, (K_X\otimes{\mathcal O}_X(S))_{x_i}$
is actually independent of the choice of the above holomorphic coordinate function $z$.

Let $D_V\, :\, V\, \longrightarrow\, V\otimes K_X\otimes{\mathcal O}_X(S)$ be a logarithmic
connection on the holomorphic vector bundle $V$. Using the Leibniz identity in
\eqref{e12} it is straightforward to deduce that the composition of homomorphisms
\begin{equation}\label{ec1}
V\, \xrightarrow{\,\ D_V\,\ }\, V\otimes K_X\otimes{\mathcal O}_X(S) \, \longrightarrow\,
(V\otimes K_X\otimes{\mathcal O}_X(S))_{x_i}\,\stackrel{\sim}{\longrightarrow}\, V_{x_i}
\end{equation}
is ${\mathcal O}_X$--linear; the above isomorphism
$(V\otimes K_X\otimes{\mathcal O}_X(S))_{x_i}\,
\stackrel{\sim}{\longrightarrow}\, V_{x_i}$ is given by the isomorphism in \eqref{pa}.
Therefore, the composition of homomorphisms in \eqref{ec1} produces a
$\mathbb C$--linear homomorphism
\begin{equation}\label{er}
{\rm Res}(D_V,\,x_i)\, :\, V_{x_i}\, \longrightarrow\, V_{x_i}\, ,
\end{equation}
which is called the \textit{residue} of $D_V$ at $x_i$; see \cite{De}.
If $\lambda_1,\, \cdots ,\, \lambda_r$ are the eigenvalues of ${\rm Res}(D_V,\,x_i)$
with multiplicity, where $r\,=\, \text{rank}(V)$, then the eigenvalues of the local
monodromy of $D_V$ around $x_i$ are
\begin{equation}\label{erm}
\exp(-2\pi\sqrt{-1}\lambda_1),\, \exp(-2\pi\sqrt{-1}\lambda_2),\, \cdots ,\, \exp(-2\pi\sqrt{-1}\lambda_r)
\end{equation}
(see \cite{De}).

We now recall another description of the logarithmic connections on $V$. Consider
the short exact sequence in \eqref{e6}
\begin{equation}\label{e14}
0\, \longrightarrow\, \text{End}(V)\, \longrightarrow\,
{\rm Diff}^1(V,\, V) \, \stackrel{\widehat{\sigma}_V}{\longrightarrow}\,
TX\otimes \text{End}(V)\, \longrightarrow\, 0
\end{equation}
for $V$. Define
$$
\text{At}(V,S)\, :=\, \widehat{\sigma}^{-1}_V(TX\otimes{\mathcal O}_X(-S))\, \subset\,
{\rm Diff}^1(V,\, V)\, ,
$$
where $TX\otimes{\mathcal O}_X(-S)\, \subset\, TX\otimes \text{End}(V)$ is
the subbundle defined as in \eqref{e13} for $V$. So from \eqref{e14} we have
the short exact sequence of holomorphic vector bundles
\begin{equation}\label{e15}
0\, \longrightarrow\, \text{End}(V)\, \longrightarrow\,
\text{At}(V,S) \, \stackrel{\sigma_V}{\longrightarrow}\,
TX\otimes{\mathcal O}_X(-S) \, \longrightarrow\, 0\, ,
\end{equation}
where $\sigma_V$ is the restriction of the projection $\widehat{\sigma}_V$ in
\eqref{e14} to the subsheaf $\text{At}(V,S)\, \subset\, {\rm Diff}^1(V,\, V)$.
A logarithmic connection on $V$ is a holomorphic splitting of the short exact sequence
in \eqref{e15}; in other words, a logarithmic connection on $V$ is a holomorphic
homomorphism of vector bundles
$$
h\, :\, TX\otimes{\mathcal O}_X(-S)\, \longrightarrow\,\text{At}(V,S)
$$
such that $\sigma_V\circ h\,=\, {\rm Id}_{TX\otimes{\mathcal O}_X(-S)}$, where
$\sigma_V$ is the homomorphism in \eqref{e15}.

Take a parabolic vector bundle $E_*\,=\,(E,\, (\{E_{i,j}\},\, \{\alpha_{i,j}\}))$ on $X$ with
parabolic structure over $S$; see \eqref{e2}, \eqref{e2a}.

A \textit{connection} on $E_*$ is a logarithmic connection $D$ on $E$, singular along
$S$, such that
\begin{enumerate}
\item $\text{Res}(D,x_i)(E_{i,j})\, \subset\, E_{i,j}$ for all $1\,\leq\, j\,\leq\, l_i$,
$1\,\leq\, i\, \leq\, n$ (see \eqref{e2}), and

\item the endomorphism of $E_{i,j}/E_{i,j+1}$ induced by $\text{Res}(D,x_i)$ coincides with
multiplication by the parabolic weight $\alpha_{i,j}$ for all $1\,\leq\, j\,\leq\, l_i$,
$1\,\leq\, i\, \leq\, n$.
\end{enumerate}
(See \cite[Section~2.2]{BL}.) Condition (2) requiring that the residues coincide with the
parabolic weights is natural in the following sense: It is shown in \cite[p.~594, Theorem 1.1]{BL} that
a parabolic vector bundle $E_*$ admits a connection if and only if the parabolic degree of every direct summand
of $E_*$ is zero --- see also the appendix of this paper and in particular Definition A.3 and Theorem A.5.

A \textit{holomorphic splitting} of the Atiyah exact sequence in \eqref{e10}
(see Definition \ref{def1}) is a holomorphic homomorphism of vector bundles
$$
\mathbf{h}\, :\, TX\otimes{\mathcal O}_X(-S)\, \longrightarrow\, {\rm At}(E_*)
$$
such that
\begin{equation}\label{e16}
\sigma\circ \mathbf{h}\,=\, {\rm Id}_{TX\otimes{\mathcal O}_X(-S)}\, ,
\end{equation}
where $\sigma$ is the projection in \eqref{e10}.

\begin{lemma}\label{lem1}
Giving a holomorphic splitting $\mathbf{h}$ of the Atiyah exact sequence for $E_*$ (see
\eqref{e16}) is equivalent to giving a logarithmic connection $D$ on $E$ satisfying the
following condition: for every $x_i\, \in\, S$,
$$
{\rm Res}(D,\, x_i)(E_{i,j})\, \subset\, E_{i,j}
$$
for all $1\, \leq\, j\, \leq\, l_i$ (see \eqref{e2}), where ${\rm Res}(D,\, x_i)$
is constructed in \eqref{er}.
\end{lemma}

\begin{proof}
Let $\mathbf{h}$ be a holomorphic splitting of the Atiyah exact sequence in \eqref{e10}. In other
words, $\mathbf{h}\, :\, TX\otimes{\mathcal O}_X(-S)\, \longrightarrow\, {\rm At}(E_*)$
is a holomorphic homomorphism such that \eqref{e16} holds. Recall from \eqref{e9} that
${\rm At}(E_*)\, \subset\, {\rm Diff}^1_P(E,\, E)$. The
composition of homomorphisms
$$
TX\otimes{\mathcal O}_X(-S)\, \stackrel{\mathbf{h}}{\longrightarrow}\, {\rm At}(E_*)\,
\hookrightarrow\, {\rm Diff}^1_P(E,\, E)
$$
will be denoted by $\widetilde{\mathbf{h}}$. Let $\widetilde{\mathbf{h}}'$ denote the
composition of homomorphisms
\begin{equation}\label{e17}
TX\otimes{\mathcal O}_X(-S)\, \stackrel{\widetilde{\mathbf{h}}}{\longrightarrow}\,
{\rm Diff}^1_P(E,\, E) \,\stackrel{\iota_0}{\hookrightarrow}\, {\rm Diff}^1(E,\, E)
\end{equation}
(see \eqref{e7} for the inclusion map $\iota_0$). Since \eqref{e16} holds, we conclude that
$$
\sigma_0\circ \widetilde{\mathbf{h}}'\,=\, {\rm Id}_{TX\otimes{\mathcal O}_X(-S)}\, ,
$$
where $\sigma_0$ is the projection in \eqref{e6}. Therefore, $\widetilde{\mathbf{h}}'$
is a logarithmic connection on $E$; this logarithmic connection on $E$
will be denoted by $D$. Since
$$
\widetilde{\mathbf{h}}'\,=\, \iota_0\circ \widetilde{\mathbf{h}}
$$
(see \eqref{e17}), the residues of the above defined logarithmic connection $D$
satisfy the following condition: for every $x_i\, \in\, S$,
$$
{\rm Res}(D,\, x_i)(E_{i,j})\, \subset\, E_{i,j}
$$
for all $1\, \leq\, j\, \leq\, l_i$.

To prove the converse, let $D$ be a logarithmic connection on $E$ such that
\begin{equation}\label{e18}
{\rm Res}(D,\, x_i)(E_{i,j})\, \subset\, E_{i,j}
\end{equation}
for all $x_i\, \in\, S$ and all $1\, \leq\, j\, \leq\, l_i$. So $D$ gives a
holomorphic homomorphism
$$
\mathbf{h}\, :\, TX\otimes{\mathcal O}_X(-S)\, \longrightarrow\,
\text{At}(E,S)
$$
(see \eqref{e15} for $\text{At}(E,S)$) such that the composition of homomorphisms
$$
TX\otimes{\mathcal O}_X(-S)\, \stackrel{\mathbf{h}}{\longrightarrow}\,
\text{At}(E,S)\, \longrightarrow\, TX\otimes{\mathcal O}_X(-S)
$$
coincides with the identity map of $TX\otimes{\mathcal O}_X(-S)$; see \eqref{e15}
for the above projection $\text{At}(E,S)\, \longrightarrow\, TX\otimes{\mathcal O}_X(-S)$.
The given condition in \eqref{e18} implies that
$$
\mathbf{h}(TX\otimes{\mathcal O}_X(-S))\, \subset\,
{\rm Diff}^1_P(E,\, E)\bigcap\text{At}(E,S)\,=\, {\rm At}(E_*)\, ;
$$
note that from \eqref{e9} we have
$$
{\rm At}(E_*)\, :=\, {\rm Diff}^1_P(E,\, E)
\bigcap \sigma^{-1}_0(TX\otimes{\mathcal O}_X(-S))\, \subset\,
\sigma^{-1}_0(TX\otimes{\mathcal O}_X(-S))\,=\, \text{At}(E,S)\, .
$$
Consequently, the homomorphism $\mathbf{h}$ gives a holomorphic splitting,
as in \eqref{e16}, of the Atiyah exact sequence for $E_*$.
\end{proof}

\begin{corollary}\label{cor1}
Giving a connection on the parabolic vector bundle $E_*$ is equivalent to giving a holomorphic
splitting $\mathbf{h}$ of the Atiyah exact sequence for $E_*$ (see
\eqref{e10}) such that the logarithmic connection $D$ on $E$ associated to
$\mathbf{h}$ (see Lemma \ref{lem1}) satisfies the following condition: for every
$x_i\, \in\, S$, the residue ${\rm Res}(D,\, x_i)$ induces the endomorphism
$\alpha_{i,j}\cdot {\rm Id}_{E_{i,j}/E_{i,j+1}}$ of the quotient space
$E_{i,j}/E_{i,j+1}$ for all $1\, \leq\, j\, \leq\, l_i$.
\end{corollary}

\begin{proof}
Note that from Lemma \ref{lem1} we know that
$$
{\rm Res}(D,\, x_i)(E_{i,j})\, \subset\, E_{i,j}
$$
for all $x_i\, \in\, S$ and all $1\, \leq\, j\, \leq\, l_i$. Therefore,
${\rm Res}(D,\, x_i)$ induces an endomorphism of the quotient space
$E_{i,j}/E_{i,j+1}$.

The result follows from Lemma \ref{lem1} and the definition of a connection on $E_*$.
\end{proof}

\subsection{A homomorphism associated to a connection}

Let $E_* \,=\, (E,\, (\{E_{i,j}\},\, \{\alpha_{i,j}\}))$ be a parabolic vector bundle
on $X$ with parabolic divisor $S$. Let $D$ be a connection on the parabolic vector bundle $E_*$.
Using $D$ we will construct a first order holomorphic differential operator
\begin{equation}\label{e19}
\mathcal{D}_0\, :\, \text{End}^P(E_*)\, \longrightarrow\, \text{End}^n(E_*)\otimes K_X
\otimes{\mathcal O}_X(S)
\end{equation}
(see \eqref{e4} and \eqref{e5} for $\text{End}^P(E_*)$ and $\text{End}^n(E_*)$ respectively).

To construct $\mathcal{D}_0$, take any holomorphic section $\Phi\, \in\, \Gamma(U,\, 
\text{End}(E))$, where $U\, \subset\, X$
is any open subset. Let
$$
A_U\, :\, \Gamma(U,\, E)\,\longrightarrow\, \Gamma(U,\, E\otimes K_X
\otimes{\mathcal O}_X(S))
$$
be the homomorphism defined by
$$
s\, \longmapsto\, D(\Phi(s))- (\Phi\otimes {\rm Id}_{K_X
\otimes{\mathcal O}_X(S)})(D(s))\, .
$$
This $A_U$ is evidently ${\mathcal O}_U$--linear. Hence we have
$$
A_U\, \in\, \Gamma(U,\, \text{End}(E)\otimes K_X\otimes{\mathcal O}_X(S))\, .
$$
{}From the properties of $D$ it follows that
$$
A_U \, \in\, \Gamma(U,\, \text{End}^n(E_*)\otimes K_X \otimes{\mathcal O}_X(S))
\, \subset\, \Gamma(U,\, \text{End}(E)\otimes K_X\otimes{\mathcal O}_X(S))
$$
if $\Phi\, \in\, \Gamma(U,\, \text{End}^P(E_*))$. The homomorphism $\mathcal{D}_0$
in \eqref{e19} is defined by $\Phi\, \longmapsto\, A_U$.

Recall from \eqref{e10} that $\text{End}^P(E_*)$ is a holomorphic subbundle of
${\rm At}(E_*)$. We will now extend $\mathcal{D}_0$ in \eqref{e19} to a
first order holomorphic differential operator
\begin{equation}\label{e21}
\mathcal{D}\ :\ {\rm At}(E_*)\ \longrightarrow\ \text{End}^n(E_*)\otimes K_X
\otimes{\mathcal O}_X(S).
\end{equation}

To construct $\mathcal{D}$, take holomorphic sections $$\Phi\, \in\, \Gamma(U,\,
{\rm At}(E_*))\ \ \text{ and }\ \ v\, \in\, \Gamma(U,\, TX\otimes{\mathcal O}_X(-S)),$$
where $U\, \subset\, X$ is any open subset. Denote
\begin{equation}\label{ew}
w\, :=\, \sigma(\Phi)\, \in\, \Gamma(U,\, TX\otimes{\mathcal O}_X(-S)),
\end{equation}
where $\sigma$ is the projection in \eqref{e10}. Take any $s\, \in\, \Gamma(U,\, E)$. So
$$
\Phi(s)\, \in\, \Gamma(U,\, E)
$$
(recall from \eqref{e9} that ${\rm At}(E_*)\, \subset\, {\rm Diff}^1(E,\, E)$), and hence
$$D(\Phi(s))\, \in\, \Gamma(U,\, E\otimes K_X\otimes{\mathcal O}_X(S))\, .$$ Therefore,
we have
\begin{equation}\label{e22}
\langle D(\Phi(s)),\, v\rangle \, \in\, \Gamma(U,\, E)\, ,
\end{equation}
where $\langle -,\, -\rangle$ is the natural duality pairing
\begin{equation}\label{dp}
(K_X\otimes{\mathcal O}_X(S))\otimes (TX\otimes{\mathcal O}_X(-S))\, \longrightarrow\,
{\mathcal O}_X\, .
\end{equation}

We have $\langle D(s),\, v\rangle \, \in\, \Gamma(U,\, E)$, so
\begin{equation}\label{e23}
\Phi(\langle D(s),\, v\rangle)\, \in\, \Gamma(U,\, E)\, ,
\end{equation}
where $\langle -,\, -\rangle$ is the pairing in \eqref{dp}.
Consider the Lie bracket of vector fields
$$
[v,\, w]\, \in\, \Gamma(U,\, TX\otimes{\mathcal O}_X(-S))\, ,
$$
where $w$ is defined in \eqref{ew}. We have
\begin{equation}\label{e24}
\langle D(s),\, [v,\, w]\rangle\, \in\, \Gamma(U,\, E)\, .
\end{equation}

Let $B_U\, :\, \Gamma(U,\, E)\, \longrightarrow\, \Gamma(U,\, E)$
be the homomorphism defined by
$$
s\, \longmapsto\, \langle D(\Phi(s)),\, v\rangle- \Phi(\langle D(s),\, v\rangle)
-\langle D(s),\, [v,\, w]\rangle
$$
(see \eqref{e22}, \eqref{e23} and \eqref{e24}). This homomorphism $B_U$ is evidently
${\mathcal O}_U$--linear, and hence we have
\begin{equation}\label{ec2}
B_U\, \in\, \Gamma(U,\, \text{End}(E))\, .
\end{equation}
It is now straightforward to check that
$$
B_U\, \in\, \Gamma(U,\, \text{End}^n(E_*))\, \subset\, \Gamma(U,\, \text{End}(E))\, .
$$

The homomorphism $\mathcal D$ in \eqref{e21} is uniquely defined by the following
property: For any open subset $U\, \subset\, X$ and sections 
$$\Phi\, \in\, \Gamma(U,\,
{\rm At}(E_*))\ \ \text{ and }\ \ v\, \in\, \Gamma(U,\, TX\otimes{\mathcal O}_X(-S)),$$
the equality
$$
\langle {\mathcal D}(\Phi),\, v\rangle\,=\, B_U
$$
holds, where $B_U$ is constructed in \eqref{ec2} from $\Phi$ and $v$ using $D$,
and $\langle -,\, -\rangle$ is the pairing in \eqref{dp}.

{}From the constructions of $\mathcal D$ and $\mathcal{D}_0$
(see \eqref{e19}) it follows immediately
that the restriction of $\mathcal{D}$ to $\text{End}^P(E_*)\, \subset\, {\rm At}(E_*)$
actually coincides with $\mathcal{D}_0$.

\section{Infinitesimal deformations and isomonodromy}

\subsection{Infinitesimal deformations}

Let $E_*$ be a parabolic vector bundle over $X$ with parabolic divisor $S$. Then the
infinitesimal deformations of $E_*$, keeping the pair $(X,\, S)$ fixed,
are parametrized by $H^1(X,\, \text{End}^P(E_*))$, where $\text{End}^P(E_*)$ is defined
in \eqref{e4} \cite{MS}. We recall that the infinitesimal deformations of the pair
$(X,\, S)$ are parametrized by $H^1(X,\, TX\otimes{\mathcal O}_X(-S))$ (see, for instance, \cite[Chapter 8,
pp. 94--95]{HM}).

\begin{lemma}\label{lem2}
The infinitesimal deformations of the triple
$$(X,\, S,\, E_*)$$ are parametrized by $H^1(X,\, {\rm At}(E_*))$, where ${\rm At}(E_*)$
is defined in \eqref{e9}.
\end{lemma}

\begin{proof} 
The proof is very similar to that of
\cite[p.~1413, Proposition 4.3]{Ch2} (see also \cite{Ch1}). Recall that \cite[p.~1413, Proposition 4.3]{Ch2}
proves that the infinitesimal deformations of the pair $(X,\, E_G)$, where
$X$ is a compact
Riemann surface structure and $E_G$ a holomorphic principal $G$-bundle over $X$ are parametrized by
the elements of the cohomology $H^1(X,\, {\rm At}(E_G))$, with 
${\rm At}(E_G))$ being the associated Atiyah bundle.

The lemma actually follows from \cite[p.~1413, Proposition 4.3]{Ch2} once we invoke the 
correspondence between the parabolic vector bundles and the orbifold bundles. This is explained 
below.

For any $x_i\,\in\, S$, let $N_i$ be the smallest positive integer such that
for all $1\, \leq\, j\, \leq\, l_i$,
$$\alpha_{i,j} \,=\, \frac{m_{i,j}}{N_i} ,$$ 
where $m_{i,j}$ are nonnegative integers; see \eqref{e2a} and Assumption \ref{asm1}.
There is a ramified Galois covering
\begin{equation}\label{gc}
\varphi\, :\, Y\, \longrightarrow\, X
\end{equation}
satisfying the following two conditions:
\begin{itemize}
\item $\varphi$ is unramified over the complement $X\setminus S$, and

\item for every $x_i\, \in\, S$ and one (hence every) point $y\, \in\, \varphi^{-1}(x_i)$,
the order of ramification of $\varphi$ at $y$ is $N_i$.
\end{itemize}
Such a covering $\varphi$ exists; see \cite[p. 26, Proposition 1.2.12]{Na} and
Assumption \ref{asm1}.

Let $\Gamma_\varphi\,=\, \text{Gal}(\varphi) \, \subset\, \text{Aut}(Y)$ be the Galois group
for the ramified covering $\varphi$, so
$X\,=\, Y/\Gamma_\varphi$. An equivariant vector bundle over $Y$ is a holomorphic
vector bundle $V\, \longrightarrow\, Y$ equipped with a lift of the action of
$\Gamma_\varphi$. In other words,
\begin{itemize}
\item $\Gamma_\varphi$ acts holomorphically on the total space
of $V$, and

\item the action of any $g\, \in\, \Gamma_\varphi$ on $V$ is a holomorphic automorphism
of the vector bundle $V$ over the automorphism $g$ of $Y$. In particular, the projection
map $V\, \longrightarrow\, Y$ is $\Gamma_\varphi$--equivariant.
\end{itemize}
There is a natural equivalence of categories between the parabolic vector bundles on $X$
whose parabolic weights at each $x_i$ are integral multiples of $\frac{1}{N_i}$ and the
equivariant vector bundles on $Y$ \cite{Bi0}, \cite{Bo1}, \cite{Bo2}.

Let $F_*$ be a parabolic vector bundle on $X$
whose parabolic weights at each $x_i$ are integral multiples of $\frac{1}{N_i}$.
The holomorphic vector bundle underlying $F_*$ will be denoted by $F$.
Let $V$ be the equivariant vector bundle on $Y$ that corresponds to $F_*$. The action
$\Gamma_\varphi$ on $V$ produces a homomorphism from $\Gamma_\varphi$ to the
group $\text{Aut}(\varphi_*V)$ of holomorphic automorphisms of the direct image
$\varphi_*V$, over the identity map of $X$. Then we have
$$
F\,=\, (\varphi_*V)^{\Gamma_\varphi}\, \subset\, \varphi_*V\, ,
$$
where $(\varphi_*V)^{\Gamma_\varphi}$ denotes the invariant part for the action
of $\Gamma_\varphi$ on $\varphi_*V$.

Take a holomorphic family of compact Riemann surfaces
equipped with $n$ ordered marked points.
Assume that this family is parametrized by $T$, and that there is a point $t_0\, \in\, T$
such that the fiber over $t_0$ is the given pair $(X,\, S)$. Then the construction
of the ramified Galois covering of $X$, done in \cite[p. 26, Proposition 1.2.12]{Na},
extends to produce a family of ramified Galois coverings of all compact Riemann surfaces
over an open neighborhood of $t_0$ in $T$.

Let $V$ be the equivariant bundle on $Y$ corresponding to the parabolic vector bundle
$E_*$ in the lemma. Then ${\rm End}^P(E_*)$ is the holomorphic vector bundle underlying the
parabolic vector bundle that corresponds to the equivariant
vector bundle $\text{End}(V)$ on $Y$. The action of
$\Gamma_\varphi$ on $Y$ produces an action of $\Gamma_\varphi$ on $TY$, and
$TX\otimes {\mathcal O}_X(-S)$ is the
holomorphic line bundle underlying the corresponding parabolic line bundle
on $X$. The actions of $\Gamma_\varphi$ on $V$ and $Y$ together produce
an action of $\Gamma_\varphi$ on $\text{At}(V)$. The Atiyah bundle
${\rm At}(E_*)$ is the holomorphic vector bundle underlying the parabolic vector bundle
corresponding to the equivariant vector bundle $\text{At}(V)$ on $Y$. This implies that
\begin{equation}\label{f1}
H^1(X,\, {\rm At}(E_*))\,=\, H^1(Y,\, {\rm At}(V))^{\Gamma_\varphi}\, .
\end{equation}

In view of\eqref{f1}, the lemma follows immediately from \cite[p.~1413, Proposition 
4.3]{Ch2}.
\end{proof}

Let
\begin{equation}\label{e8}
0\,=\, H^0(X,\, TX\otimes{\mathcal O}_X(-S))
\,\longrightarrow\, H^1\left(X,\, \text{End}^P(E_*)\right)\, \stackrel{p_1}{\longrightarrow}\,
H^1(X,\, {\rm At}(E_*))
\end{equation}
$$
\stackrel{p_2}{\longrightarrow}\,
H^1(X,\, TX\otimes{\mathcal O}_X(-S)) \, \longrightarrow\,
H^2\left(X,\, \text{End}^P(E_*)\right)\,=\,0
$$
be the long exact sequence of cohomologies associated to the short exact sequence
in \eqref{e10}; we have $H^0(X,\, TX\otimes{\mathcal O}_X(-S))\,=\,0$ by
Assumption \ref{asm1}.
The projection $p_2$ in \eqref{e8} is the forgetful map that sends an
infinitesimal deformation of the triple $(X,\, S,\, E_*)$ to the
infinitesimal deformation of the pair $(X,\, S)$ obtained from it by simply
forgetting $E_*$. The injective homomorphism $p_1$ in \eqref{e8} sends an infinitesimal
deformation of $E_*$ to the infinitesimal deformation of $(X,\, S,\, E_*)$ obtained from
it by keeping the pair $(X,\, S)$ fixed.

\begin{lemma}\label{los}
Assume that the parabolic vector bundle $E_*$ has a connection $D$ such that
the local monodromy of $D$ around each point of $S$ is semisimple (meaning diagonalizable).
Then the local monodromy of every connection on $E_*$ around each point of $S$ is also semisimple.
\end{lemma}

\begin{proof}
Take a ramified Galois covering
$$
\varphi\, :\, Y\, \longrightarrow\, X,
$$
as in \eqref{gc}, satisfying the following two conditions:
\begin{itemize}
\item $\varphi$ is unramified over the complement $X\setminus S$, and

\item for every $x_i\, \in\, S$ and one (hence every) point $y\, \in\, \varphi^{-1}(x_i)$,
the order of ramification of $\varphi$ at $y$ is $N_i$.
\end{itemize}
Let $V$ be the equivariant vector bundle on $Y$ that corresponds to the
parabolic vector bundle $E_*$. Then $D$ corresponds to
an equivariant holomorphic connection on $V$. The space of all connections
on $E_*$ is an affine space for the vector space 
$H^0(X,\,
\text{End}^n(E_*)\otimes K_X\otimes{\mathcal O}_X(S))$, where $\text{End}^n(E_*)$
is defined in \eqref{e5}. On the other hand the space of all equivariant holomorphic connections
on $V$ is an affine space for the vector space $H^0(Y,\,
\text{End}(V)\otimes K_Y)^\Gamma$, where $\Gamma$ is the Galois group for the covering map $\varphi$. 

Since we have
$$H^0(X,\, \text{End}^n(E_*)\otimes K_X\otimes{\mathcal O}_X(S))\,=\, H^0(Y,\,
\text{End}(V)\otimes K_Y)^\Gamma,$$
we conclude that every connection $D'$ on $E_*$ is given by an equivariant connection on $V$. This implies
that the order of the local monodromy of $D'$ around any point $x_i\, \in\, S$ is a sub-multiple of $N_i$.
This implies that the local monodromy of $D'$ around every point of $S$ is semisimple.
\end{proof}

Let $D$ be a connection on the parabolic vector bundle $E_*$.
We assume that the local monodromy of $D$ around each point of $S$ is semisimple. As mentioned
above, the space of all connections on $E_*$ is an affine space for the vector space $H^0(X,\,
\text{End}^n(E_*)\otimes K_X\otimes{\mathcal O}_X(S))$, where $\text{End}^n(E_*)$
is defined in \eqref{e5}. This implies that the infinitesimal deformations of the 
connection $D$, keeping $(X,\, S,\, E_*)$ fixed, are parametrized by
$H^0(X,\, \text{End}^n(E_*)\otimes K_X\otimes{\mathcal O}_X(S))$.

Let ${\mathcal A}_\bullet$ be the following two-term complex of sheaves on $X$
\begin{equation}\label{e25}
{\mathcal A}_\bullet\,\,:\,\, {\mathcal A}_0\,=\, \text{End}^P(E_*)\,
\stackrel{\mathcal{D}_0}{\longrightarrow}\, {\mathcal A}_1\,=\,\text{End}^n(E_*)
\otimes K_X\otimes{\mathcal O}_X(S)\, ,
\end{equation}
where $\mathcal{D}_0$ is the differential operator in \eqref{e19}, and
${\mathcal A}_i$ is at the $i$-th position.

The following lemma is a standard fact.

\begin{lemma}\label{lem3}
The infinitesimal deformations of the pair $(E_*,\, D)$, keeping $(X,\, S)$ fixed,
are parametrized by the first hypercohomology group
${\mathbb H}^1({\mathcal A}_\bullet)$, where ${\mathcal A}_\bullet$ is the
complex in \eqref{e25}.
\end{lemma}

The complex ${\mathcal A}_\bullet$ in \eqref{e25} fits in the following short
exact sequence of complexes of sheaves on $X$
$$
\begin{matrix}
&&0&& 0\\
&&\Big\downarrow && \Big\downarrow \\
&&0 & \longrightarrow & \text{End}^n(E_*)\otimes K_X\otimes{\mathcal O}_X(S)\\
&&\Big\downarrow && \,\,\,\,\,\,\Big\downarrow{\rm id} \\
{\mathcal A}_\bullet &: & \text{End}^P(E_*) &
\stackrel{\mathcal{D}_0}{\longrightarrow} & \text{End}^n(E_*)\otimes K_X
\otimes{\mathcal O}_X(S)\\
&&\,\,\,\,\,\,\Big\downarrow {\rm id}&& \Big\downarrow \\
&& \text{End}^P(E_*)& \longrightarrow & 0\\
&&\Big\downarrow && \Big\downarrow \\
&&0&& 0
\end{matrix}
$$
Let
\begin{equation}\label{e26}
\longrightarrow\, H^0(X,\, \text{End}^n(E_*)\otimes K_X
\otimes{\mathcal O}_X(S))\, \stackrel{\alpha_1}{\longrightarrow}\, 
{\mathbb H}^1({\mathcal A}_\bullet)\, \stackrel{\alpha_2}{\longrightarrow}\,
H^1\left(X,\, \text{End}^P(E_*)\right)\,\longrightarrow
\end{equation}
be the long exact sequence of hypercohomology groups associated to the above
short exact sequence of complexes. The homomorphism
$\alpha_2$ in \eqref{e26} sends an infinitesimal deformation of the pair $(E_*,\, D)$
to the infinitesimal deformation of $E_*$ obtained from it by simply forgetting the
connection. The homomorphism $\alpha_1$ in \eqref{e26} sends an infinitesimal deformation
of the connection $D$ to the infinitesimal deformation of the pair $(E_*,\, D)$ obtained
from it by keeping the parabolic vector bundle $E_*$ fixed.

Let ${\mathcal B}_\bullet$ denote the following two-term complex of sheaves on $X$
\begin{equation}\label{e27}
{\mathcal B}_\bullet\,\,:\,\, {\mathcal B}_0\,=\, {\rm At}(E_*)\,
\stackrel{\mathcal{D}}{\longrightarrow}\, {\mathcal B}_1\,=\,\text{End}^n(E_*)
\otimes K_X\otimes{\mathcal O}_X(S)\, ,
\end{equation}
where $\mathcal{D}$ is the homomorphism in \eqref{e21}, and
${\mathcal B}_i$ is at the $i$-th position.

\begin{lemma}\label{lem4}
The infinitesimal deformations of the quadruple $$(X,\, S,\, E_*,\, D)$$
are parametrized by the first hypercohomology group
${\mathbb H}^1({\mathcal B}_\bullet)$, where ${\mathcal B}_\bullet$ is the
complex in \eqref{e27}.
\end{lemma}

\begin{proof}
The proof is very similar to the proof of
\cite[p.~1415, Proposition 4.4]{Ch2} (see also \cite{Ch1}). Recall that \cite[p.~1415, Proposition 4.4]{Ch2} proves that the infinitesimal deformations of the triple $(X,\, E_G,\, \phi)$ (where $X$ is a
compact Riemann surface structure, $E_G$ is a holomorphic principal $G$-bundle over $X$ and $\phi$ is a holomorphic connection on $E_G$) are parametrized by
the elements of the hypercohomology group
${\mathbb H}^1(X,\, {\mathcal C}_\bullet)$, with ${\mathcal C}_\bullet$ being the analogous of the
complex ${\mathcal B}_\bullet$ in the non parabolic case. In fact 
the proof of Lemma \ref{lem4} can be deduced directly from \cite[p.~1415, Proposition 4.4]{Ch2}, as done in
the proof of Lemma \ref{lem2}. This is elaborated below.

Take the ramified Galois covering $(Y,\, \varphi)$ in \eqref{gc}. As in the proof of Lemma 
\ref{lem2}, $V$ denotes the equivariant vector bundle on $Y$ corresponding to the 
parabolic vector bundle $E_*$. The connection $D$ on $E_*$ corresponds to a 
$\Gamma_\varphi$--invariant holomorphic connection on $V$, where $\Gamma_\varphi$ is, as 
before, the Galois group for $\varphi$; this $\Gamma_\varphi$--invariant holomorphic 
connection on $V$ will be denoted by $D'$. Let
\begin{equation}\label{f2}
{\mathcal B}'_\bullet\, :\, {\mathcal B}'_0\,:=\,\text{At}(V)\,
\stackrel{\mathcal{D}'}{\longrightarrow}\, {\mathcal B}_1\,:=\,\text{End}(V)\otimes K_Y
\end{equation}
be the complex in \cite[p.~1415, Proposition 4.4]{Ch2} for $D'$; it is the same complex
as in \eqref{e27} when there is no parabolic structure (meaning the parabolic structure
is trivial). We note that the differential
operator $\mathcal{D}'$ in \eqref{f2} is $\Gamma_\varphi$--equivariant, because the connection
$D'$ on $V$ is $\Gamma_\varphi$--invariant. The holomorphic vector bundle $\text{End}^n(E_*)
\otimes K_X\otimes{\mathcal O}_X(S)$ coincides with the holomorphic vector bundle underlying
the parabolic vector bundle that corresponds to the
$\Gamma_\varphi$--equivariant bundle $\text{End}(V)\otimes K_Y$ on $Y$.
It was noted in the proof of Lemma \ref{lem2} that
${\rm At}(E_*)$ is the holomorphic vector bundle underlying the parabolic vector bundle
corresponding to the equivariant bundle $\text{At}(V)$ on $Y$. Moreover, the operator
differential $\mathcal{D}$ in \eqref{e27} coincides with the one given by $\mathcal{D}'$
on the $\Gamma_\varphi$--invariant part of the direct image. These together imply
that
\begin{equation}\label{f3}
{\mathbb H}^1({\mathcal B}_\bullet)\,=\, {\mathbb H}^1
({\mathcal B}'_\bullet)^{\Gamma_\varphi},
\end{equation}
where ${\mathbb H}^1({\mathcal B}'_\bullet)^{\Gamma_\varphi}\, \subset\,
{\mathbb H}^1({\mathcal B}'_\bullet)$ is the invariant part for the action
of $\Gamma_\varphi$.

In view of \eqref{f3}, the lemma follows from \cite[p.~1415, Proposition 4.4]{Ch2}. 
\end{proof}

The complex ${\mathcal B}_\bullet$ in \eqref{e27} fits in the following
short exact sequence of complexes of sheaves on $X$
$$
\begin{matrix}
&&0&& 0\\
&&\Big\downarrow && \Big\downarrow \\
{\mathcal A}_\bullet & : &\text{End}^P(E_*) & \stackrel{\mathcal{D}_0}{\longrightarrow}
& \text{End}^n(E_*)\otimes K_X\otimes{\mathcal O}_X(S)\\
&&\Big\downarrow && \,\,\,\,\,\,\Big\downarrow{\rm id} \\
{\mathcal B}_\bullet &: & {\rm At}(E_*) &
\stackrel{\mathcal{D}}{\longrightarrow} & \text{End}^n(E_*)\otimes K_X
\otimes{\mathcal O}_X(S)\\
&&\Big\downarrow && \Big\downarrow \\
&& TX\otimes{\mathcal O}_X(-S) & \longrightarrow & 0\\
&&\Big\downarrow && \Big\downarrow \\
&&0&& 0
\end{matrix}
$$
where the vertical exact sequence in the left is the one in \eqref{e10}; see
\eqref{e25} and \eqref{e27} for ${\mathcal A}_\bullet$ and ${\mathcal B}_\bullet$
respectively. Let
\begin{equation}\label{e28}
\longrightarrow\, {\mathbb H}^1({\mathcal A}_\bullet)\, \stackrel{\beta_1}{\longrightarrow}\, 
{\mathbb H}^1({\mathcal B}_\bullet)\, \stackrel{\beta_2}{\longrightarrow}\,
H^1(X,\, TX\otimes{\mathcal O}_X(-S))\,\longrightarrow
\end{equation}
be the long exact sequence of hypercohomology groups associated to the above
short exact sequence of complexes. The homomorphism
$\beta_2$ in \eqref{e28} sends an infinitesimal deformation of the quadruple $(X,\, S,\, E_*,\, D)$
to the infinitesimal deformation of $(X,\, S)$ obtained from it by simply forgetting the
pair $(E_*,\, D)$. The homomorphism $\beta_1$ in \eqref{e28} sends an infinitesimal deformation
of the pair $(E_*,\, D)$ to the infinitesimal deformation of the
quadruple $(X,\, S,\, E_*,\, D)$ obtained
from it by keeping the pair $(X,\, S)$ fixed.

\subsection{Character variety and the monodromy map}

Fix an integer $r\, \geq\, 1$. Fix a base point $b_0\, \in\, X\setminus S$, and denote by 
$\widetilde{\rm Hom}(\pi_1(X\setminus S,\, b_0),\, {\rm GL}(r, {\mathbb C}))$ the space of 
all homomorphisms from $\pi_1(X\setminus S,\, b_0)$ to ${\rm GL}(r, {\mathbb C})$. Given 
any $\rho\,\in\, \widetilde{\rm Hom}(\pi_1(X\setminus S,\, b_0),\, {\rm GL}(r, {\mathbb 
C}))$, we consider ${\mathbb C}^r$ as a $\pi_1(X\setminus S,\, b_0)$--module by combining 
$\rho$ with the standard $r$--dimensional representation of ${\rm GL}(r, {\mathbb C})$. We 
recall that $\rho$ is called \textit{completely reducible} if the $\pi_1(X\setminus S,\, 
b_0)$--module ${\mathbb C}^r$ corresponding to $\rho$ is a direct sum of irreducible 
$\pi_1(X\setminus S,\, b_0)$--modules. Let
$$
{\rm Hom}(\pi_1(X\setminus S,\, b_0),\, {\rm 
GL}(r, {\mathbb C})) \, \subset\, \widetilde{\rm Hom}(\pi_1(X\setminus S,\, b_0),\, {\rm 
GL}(r, {\mathbb C}))
$$
be the space of all completely reducible representations. The 
adjoint action of ${\rm GL}(r, {\mathbb C})$ on itself produces an action of ${\rm GL}(r, 
{\mathbb C})$ on ${\rm Hom}(\pi_1(X\setminus S,\, b_0),\, {\rm GL}(r, {\mathbb C}))$. The 
quotient space
\begin{equation}\label{e29} {\mathcal R}_X(r)\,=\, {\rm 
Hom}(\pi_1(X\setminus S,\, b_0),\, {\rm GL}(r, {\mathbb C}))/ {\rm GL}(r, {\mathbb C}) 
\end{equation}
is a normal quasi-projective variety defined over $\mathbb C$. (See 
\cite{Si}, \cite{LS}, and references therein, for ${\mathcal R}_X(r)$.) It should be noticed that
the above quotient ${\mathcal R}_X(r)$ is well defined as a GIT quotient of the complex affine variety of representations.

For another base point $b'_0\,\in\, X\setminus S$, the two groups $\pi_1(X\setminus S,\, b_0)$
and $\pi_1(X\setminus S,\, b'_0)$ are identified up to inner automorphisms, and hence
${\mathcal R}_X(r)$ does not depend on the choice of the base point $b_0$. The
complex structure of $X$ does not play any role in the construction of ${\mathcal R}_X(r)$;
the space ${\mathcal R}_X(r)$ depends only on the topological surface underlying
$X\setminus S$.

A connection $D$ on a parabolic vector bundle $E_*$ is called \textit{completely 
reducible} if it is a direct sum of irreducible logarithmic connections.

Fix the dimensions of the subspaces $E_{i,j}$ and fix the parabolic weights 
$\alpha_{i,j}$; see \eqref{e2}, \eqref{e2a}. Let ${\mathcal M}_X(r)$ denote the moduli 
space of pairs $(E_*,\, D)$, where $E_*$ is a parabolic vector bundle of rank $r$ on $X$ 
having the given parabolic type and $D$ is a completely reducible connection on $E_*$; see 
\cite{In}, \cite{IIS1}, \cite{IIS2}, \cite{Iw}, \cite{BBP} and references therein for the 
moduli space.

Since $X$ is a Riemann surface any holomorphic connection on a holomorphic vector bundle over $X$
is necessarily flat. Therefore, any logarithmic connection on a holomorphic vector bundle over $X$
is flat. Hence we can associate to any logarithmic connection on a holomorphic vector bundle over $X$ 
a monodromy representation (up to inner conjugacy). Consequently, we have a
holomorphic map
\begin{equation}\label{e30}
\mathbb{M}_X\ :\ {\mathcal M}_X(r)\ \longrightarrow\ {\mathcal R}_X(r),
\end{equation}
where ${\mathcal R}_X(r)$ is constructed in \eqref{e29}, that sends any connection to
its monodromy.

Let
\begin{equation}\label{e31}
\varpi\, :\, X_T\, \longrightarrow\, T
\end{equation}
be a holomorphic family of compact connected Riemann surfaces parametrized by a
simply connected complex manifold $T$. For any $t\, \in\, T$, the
fiber $\varpi^{-1}(t)$ will be denoted by $X_t$. For each $1\, \leq\, i\, \leq\, n$, let
\begin{equation}\label{phii}
\phi_i\,\,:\,\, T\, \longrightarrow\, X_T
\end{equation}
be a holomorphic section such that $\phi_i(T)\bigcap \phi_j(T)\, =\,\emptyset$ for
all $i\, \not=\, j$. Fix a base point $t_0\, \in\, T$. Denote $X_{t_0}$ by $X$,
and also denote $\phi_i(t_0)$ by $x_i$ for every $1\, \leq\, i\, \leq\, n$. As before,
denote the subset $\{x_1,\, \cdots,\, x_n\}\, \subset X$ by $S$. For any $t\, \in\, T$,
the subset $\{\phi_1(t),\, \cdots,\, \phi_n(t)\}\, \subset X_t\,:=\, \varpi^{-1}(t)$
will be denoted by $S_t$.

Since the parameter space $T$ is simply connected, $\pi_1(X\setminus S,\, b_0)$ and
$\pi_1(X_t\setminus S_t,\, b_t)$ are identified up to inner automorphisms. This implies
that the character variety ${\mathcal R}_{X_t}(r)$ constructed as in \eqref{e29} is
canonically identified with ${\mathcal R}_{X}(r)$.

Let
\begin{equation}\label{e33}
{\mathcal M}^T(r)\ \longrightarrow\ T
\end{equation}
be the relative moduli space of
parabolic vector bundles with connection for the family $X_T$ in \eqref{e31}. In view of the
above observation that ${\mathcal R}_{X_t}(r)$ is
identified with ${\mathcal R}_{X}(r)$, the monodromy maps $\mathbb{M}_X$ in \eqref{e30}
for points of $T$ actually fit together to produce a holomorphic map
\begin{equation}\label{e32}
\mathbb{M}\ :\ {\mathcal M}^T(r)\ \longrightarrow\ {\mathcal R}_X(r).
\end{equation}

Let $E^T_*$ be a holomorphic family of parabolic vector bundles on $X_T$, and let
$D^T$ be a relative connection $E^T_*$. In other words, the pair $(E^T_*,\, D^T)$
corresponds to a holomorphic section
\begin{equation}\label{Psi}
\Psi\,:\, T\, \longrightarrow\, {\mathcal M}^T(r)
\end{equation}
of the holomorphic family of moduli spaces in \eqref{e33}.

A holomorphic family of the above type $(E^T_*,\, D^T)$, of parabolic vector bundles with connection, is called 
isomonodromic if the composition $\mathbb{M}\circ\Psi$ is a constant map, where $\mathbb{M}$ is the 
monodromy map in \eqref{e32} and $\Psi$ is the map in \eqref{Psi}. This condition of being isomonodromic 
defines a holomorphic foliation on ${\mathcal M}^T(r)$ which is transversal to the holomorphic foliation 
given by the projection ${\mathcal M}^T(r)\, \longrightarrow\, T$ in \eqref{e33}. This transversality comes 
from the logarithmic Riemann--Hilbert correspondence \cite{De}.

In other words, the 
direct sum of these two distributions coincides with the full tangent bundle $T{\mathcal 
M}^T(r)\, \longrightarrow\,{\mathcal M}^T(r)$, on the smooth locus of ${\mathcal M}^T(r)$.

Let $D$ be a connection on a parabolic vector bundle $E_*$ over $X$ with parabolic 
structure over $S$.
We assume that the local monodromy of $D$ around each point of $S$ is semisimple.
The foliation given by isomonodromy produces a homomorphism from the 
space of infinitesimal deformations of the pair $(X,\, S)$ to the space of infinitesimal 
deformations of the quadruple $(X,\, S,\, E_*,\, D)$. In view of Lemma \ref{lem4}, this 
homomorphism given by isomonodromy --- from the space of infinitesimal deformations of $(X,\, S)$
to the space of  infinitesimal deformations of $(X,\, S,\, E_*,\, D)$ --- is given by a homomorphism
\begin{equation}\label{e34}
\gamma\ :\ H^1(X,\, TX\otimes {\mathcal O}_X(-S))\ \longrightarrow\
{\mathbb H}^1({\mathcal B}_\bullet),
\end{equation}
where ${\mathcal B}_\bullet$ is the complex in \eqref{e27}.

{}From Corollary \ref{cor1} we know that the connection $D$ produces a homomorphism
$${\mathbf h}\ :\ TX\otimes{\mathcal O}_X(-S)\ \longrightarrow\ {\rm At}(E_*)$$
such that $\sigma\circ \mathbf{h}\,=\, {\rm Id}_{TX\otimes{\mathcal O}_X(-S)}$, where $\sigma$ is the
projection in \eqref{e10}. This homomorphism ${\mathbf h}$ produces a homomorphism of complexes of sheaves on $X$
\begin{equation}\label{e35}
\begin{matrix}
&& TX\otimes{\mathcal O}_X(-S)& \longrightarrow & 0\\
&& \,\,\, \Big\downarrow{\mathbf h} && \Big\downarrow\\
{\mathcal B}_\bullet &: & {\rm At}(E_*) & \stackrel{\mathcal{D}}{\longrightarrow}
& \text{End}^n(E_*)\otimes K_X \otimes{\mathcal O}_X(S) 
\end{matrix}
\end{equation}
We note that
\begin{equation}\label{e35b}
\mathcal{D}\circ{\mathbf h}\ =\ 0
\end{equation}
because the connection $D$ is flat. The homomorphism of
complexes in \eqref{e35} produces a homomorphism of hypercohomology groups
\begin{equation}\label{e36} \delta\ :\ H^1(X,\, 
TX\otimes {\mathcal O}_X(-S))\ \longrightarrow\ {\mathbb H}^1({\mathcal B}_\bullet).
\end{equation}

\begin{lemma}\label{lem5}
The homomorphism $\gamma$ in \eqref{e34} given by isomonodromy coincides with the
homomorphism $\delta$ constructed in \eqref{e36}.
\end{lemma}

\begin{proof}
The proof of Lemma \ref{lem5} is similar to the proof of \cite[p.~1417, Proposition 5.1]{Ch2} (see also 
\cite{Ch1}). Recall that \cite[p.~1417, Proposition 5.1]{Ch2} describes the isomonodromy lifting from the 
space of infinitesimal deformations of compact Riemann surfaces to the space of infinitesimal deformations 
of the triple $(X,\, E_G,\, \phi)$ (where $X$ is a compact Riemann surface structure, $E_G$ is a 
holomorphic principal $G$--bundle over $X$ and $\phi$ is a holomorphic connection on $E_G$). In fact, as 
done in the proofs of Lemma \ref{lem2} and Lemma \ref{lem4}, the lemma can be directly deduced from 
\cite[p.~1417, Proposition 5.1]{Ch2} using the correspondence between the parabolic vector bundles and the 
orbifold vector bundles. We omit the details.
\end{proof}

\section{Infinitesimal deformations of parabolic opers}\label{se4}

We recall from \cite{BDP} the definition of parabolic ${\rm SL}(r, {\mathbb C})$-opers.

Consider the subset $S\, \subset\, X$ in \eqref{e1}; assume that the integer
$n\,=\, \# S$ is even. Fix a holomorphic line bundle ${\mathbb L}$ on $X$
of degree $-\frac{n}{2}$ such that
${\mathbb L}^{^{\otimes 2}}$ is isomorphic to ${\mathcal O}_X(-S)$; also fix a
holomorphic isomorphism
$$
\varphi_0\ :\ {\mathbb L}^{^{\otimes 2}}\ \longrightarrow\ {\mathcal O}_X(-S)
$$
of line bundles. Fix a holomorphic line bundle $K^{1/2}_X$ on $X$ of degree
$\text{genus}(X)-1$ such that $\left(K^{1/2}_X\right)^{\otimes 2}$ is isomorphic to
$K_X$, in other
words, $K^{1/2}_X$ is a theta characteristic on $X$; fix a holomorphic isomorphism
$$
I_X\ :\ \left(K^{1/2}_X\right)^{\otimes 2}\ \longrightarrow\ K_X
$$
of line bundles. Since
$$
H^1(X,\, \text{Hom}(K^{-1/2}_X\otimes {\mathbb L},\,
K^{1/2}_X\otimes {\mathbb L}))\,=\, H^1(X,\, K_X)\,=\, H^0(X,\, {\mathcal O}_X)^*
\,=\, \mathbb C
$$
(the isomorphism $H^1(X,\, K_X)\,=\, H^0(X,\, {\mathcal O}_X)^*$ is given by
Serre duality), we have a nontrivial extension of $K^{-1/2}_X\otimes {\mathbb L}$
by $K^{1/2}_X\otimes {\mathbb L}$
\begin{equation}\label{e38}
0\, \longrightarrow\, K^{1/2}_X\otimes {\mathbb L}\, \longrightarrow\, \mathcal{E}
\, \stackrel{q}{\longrightarrow}\, K^{-1/2}_X\otimes {\mathbb L} \, \longrightarrow\,0
\end{equation}
corresponding to $1\, \in\, H^1(X,\, \text{Hom}(K^{-1/2}_X\otimes {\mathbb L},\,
K^{1/2}_X\otimes {\mathbb L}))\,=\, \mathbb C$.

We will put a parabolic structure on the rank two vector bundle $\mathcal{E}$ in \eqref{e38}.
Fix a function
\begin{equation}\label{e37}
\mathbf{c}\, :\, S\, \longrightarrow\, \{t\, \in\, {\mathbb Z}\, \mid\, t\, \geq\, 2\}\, ;
\end{equation}
the integer ${\mathbf c}(x_i)$ will also be denoted by $c_i$.

Equip the vector bundle $\mathcal{E}$ in \eqref{e38} with the following parabolic structure
over $S$: For any $x_i\,\in\, S$, the quasi-parabolic filtration of $\mathcal{E}_{x_i}$ is
$$
(K^{1/2}_X\otimes {\mathbb L})_{x_i}\, \subset\, \mathcal{E}_{x_i}\, ,
$$
where $K^{1/2}_X\otimes {\mathbb L}$ is the line subbundle in \eqref{e38}. The parabolic 
weight of $(K^{1/2}_X\otimes {\mathbb L})_{x_i}$ is $\frac{2c_i-1}{2c_i}$ and the 
parabolic weight of $\mathcal{E}_{x_i}$ is $\frac{1}{2c_i}$ (see \eqref{e37}). The 
resulting parabolic vector bundle of rank two on $X$ will be denoted by $\mathcal{E}_*$. 
The parabolic degree of $\mathcal{E}_*$ is zero. We note that the parabolic exterior 
product $\bigwedge^2 \mathcal{E}_*$ of $\mathcal{E}_*$ is the trivial holomorphic line 
bundle on $X$ equipped with the trivial holomorphic structure; see \cite[Section 5]{BDP}, 
\cite{Bi} for the parabolic exterior product and the parabolic symmetric product.

\begin{remark}\label{ros}
It can be shown that the above rank two parabolic vector bundle $\mathcal{E}_*$ is indecomposable.
To prove this assume that $\mathcal{E}_*\,=\, L_*\oplus L'_*$, where $L_*$ and $L'_*$ are parabolic
line bundles. Suppose that $\textrm{\mbox{par-deg}}(L_*)\, \geq\, \textrm{\mbox{par-deg}}(L'_*)$. Consider the composition
of homomorphism
$$
L_*\, \hookrightarrow\, \mathcal{E}_* \, \stackrel{q}{\longrightarrow}\, (K^{-1/2}_X\otimes 
{\mathbb L})_* ,
$$
where $q$ is the projection in \eqref{e38} and $(K^{-1/2}_X\otimes {\mathbb L})_*$ is the quotient line bundle
$K^{-1/2}_X\otimes {\mathbb L}$ in \eqref{e38} equipped with the parabolic structure induced by
$\mathcal{E}_*$. This composition of homomorphisms vanishes identically, because
$\textrm{\mbox{par-deg}}(L_*)\, > \,\textrm{\mbox{par-deg}}((K^{-1/2}_X\otimes {\mathbb L})_*)$. This implies that 
the short exact sequence in \eqref{e38} splits. But the short exact sequence in \eqref{e38} does not split.
In view of this contradiction we conclude that the parabolic vector bundle $\mathcal{E}_*$ is indecomposable.
\end{remark}

We note that a connection on a parabolic vector bundle $V_*$ induces connections on
$\bigwedge^i V_*$ and $\text{Sym}^{i}(V_*)$ for every $i$.

Take any integer $r\, \geq\, 2$. Let $\text{Sym}^{r-1}(\mathcal{E}_*)$ be the parabolic 
symmetric product of $\mathcal{E}_*$. We have 
$\text{rank}(\text{Sym}^{r-1}(\mathcal{E}_*))\,=\, r$, and the parabolic exterior 
product $\bigwedge^r\text{Sym}^{r-1}(\mathcal{E}_*)$ is the trivial holomorphic line 
bundle on $X$ equipped with the trivial parabolic structure.

A \textit{parabolic} ${\rm SL}(r,{\mathbb C})$--\textit{oper} on $X$ is a connection $D$ on the parabolic
vector bundle $\text{Sym}^{r-1}(\mathcal{E}_*)$ such that the connection on $\bigwedge^r
\text{Sym}^{r-1}(\mathcal{E}_*)$ induced by $D$ coincides with the connection on ${\mathcal O}_X$
given by the de Rham differential $d$; see \cite[p.~511, Definition 5.2]{BDP}.

\begin{lemma}\label{os}
For any parabolic ${\rm SL}(r,{\mathbb C})$--oper connection $D$, the local monodromy of $D$
around any point of $S$ is semisimple.
\end{lemma}

\begin{proof}
First note that the parabolic $\mathcal{E}_*$
admits a connection, because $\mathcal{E}_*$ is indecomposable (this was shown in Remark \ref{ros}) and its parabolic degree is zero
\cite[p. 594, Theorem 1.1]{BL}. Since the two parabolic weights at every $x_i\, \in\, S$ do not differ by an
integer, we conclude that for any connection on the parabolic $\mathcal{E}_*$ the local monodromy around any
point of $S$ is semisimple (see \eqref{erm}). Therefore, the connection on $\text{Sym}^{r-1}(\mathcal{E}_*)$
induced by a connection on $\mathcal{E}_*$ also has the property that the local monodromy around
every point of $S$ is semisimple. From this it follows that every connection on
$\text{Sym}^{r-1}(\mathcal{E}_*)$ has the property that the local monodromy around every point of $S$ is semisimple;
see Lemma \ref{los}.
\end{proof}

The subbundle $K^{1/2}_X\otimes {\mathbb L}\, \subset\, \mathcal{E}$ in \eqref{e38} equipped
with the induced parabolic structure will be denoted by $F_*$. So the holomorphic
line bundle underlying $F_*$ is $K^{1/2}_X\otimes {\mathbb L}$, and the parabolic
weight of $F_*$ at any $x_i\, \in\, S$ is $\frac{2c_i-1}{2c_i}$. The parabolic subbundle
\begin{equation}\label{ef}
F_*\ \stackrel{\iota_F}{\hookrightarrow}\ \mathcal{E}_*
\end{equation}
produces a filtration of parabolic subbundles of $\text{Sym}^{r-1}(\mathcal{E}_*)$ in the
following way. For any $0\,\leq\, j\, \leq\, r-1$, consider the parabolic vector
bundle $(F_*)^{\otimes (r-1-j)}\otimes \text{Sym}^{j}(\mathcal{E}_*)$; by definition,
$(F_*)^{\otimes 0}$ and $\text{Sym}^{0}(\mathcal{E}_*)$
coincide with the trivial holomorphic line bundle with the
trivial parabolic structure (see \cite{Bi}, \cite{BDP} for the tensor product of parabolic
vector bundles). So the rank of $(F_*)^{\otimes (r-1-j)}\otimes \text{Sym}^{j}(\mathcal{E}_*)$
is $j+1$. This parabolic vector bundle $(F_*)^{\otimes (r-1-j)}\otimes
\text{Sym}^{j}(\mathcal{E}_*)$
will be denoted by ${\mathcal F}^{(j+1)}_*$ (since $j+1$ is its rank). We note that
\begin{equation}\label{e39}
{\mathcal F}^{(1)}_*\, \subset\, {\mathcal F}^{(2)}_*\, \subset\, \cdots
\, \subset\, {\mathcal F}^{(r-1)}_*\, \subset\, {\mathcal F}^{(r)}_*\,=\, 
\text{Sym}^{r-1}(\mathcal{E}_*)
\end{equation}
is a filtration of parabolic subbundles of $\text{Sym}^{r-1}(\mathcal{E}_*)$. For any
$1\, \leq\, i\, \leq\, r-1$, the inclusion map
$$
{\mathcal F}^{(i)}_*\, \hookrightarrow\, {\mathcal F}^{(i+1)}_*
$$
in \eqref{e39} is constructed, in a straightforward way, using the inclusion map $F_*\,
\hookrightarrow\,\mathcal{E}_*$ in \eqref{ef} together with the natural projection
$\mathcal{E}_*\otimes\text{Sym}^{j}(\mathcal{E}_*)\, \longrightarrow\,
\text{Sym}^{j+1}(\mathcal{E}_*)$. More, precisely, we have
$$
(\iota_F)^{\otimes (r-i)}\,\, :\,\,(F_*)^{\otimes (r-i)}
\, \hookrightarrow\,(F_*)^{\otimes (r-i-1)}\otimes\mathcal{E}_* ,\,
$$
where $\iota_F$ is the inclusion map in \eqref{ef}. This implies that
$$
{\mathcal F}^{(i)}_*\,=\, (F_*)^{\otimes (r-i)}\otimes \text{Sym}^{i-1}(\mathcal{E}_*)
\, \hookrightarrow\,(F_*)^{\otimes (r-i-1)}\otimes\mathcal{E}_*\otimes
\text{Sym}^{i-1}(\mathcal{E}_*)
$$
$$
\longrightarrow\,
(F_*)^{\otimes (r-i-1)}\otimes \text{Sym}^{i}(\mathcal{E}_*)
\,=\, {\mathcal F}^{(i+1)}_*\,.
$$
The above composition of homomorphisms produces the inclusion map ${\mathcal F}^{(i)}_*\,
\subset\, {\mathcal F}^{(i+1)}_*$ in \eqref{e39}.

For any $1\, \leq\, j\, \leq\, r$, the holomorphic vector bundle of rank $j$ underlying the
parabolic vector bundle ${\mathcal F}^{(j)}_*$ will be denoted by ${\mathcal F}^{(j)}$. So the
filtration in \eqref{e39} produce a filtration of holomorphic subbundles 
\begin{equation}\label{e40}
{\mathcal F}^{(1)}\, \subset\, {\mathcal F}^{(2)}\, \subset\, \cdots
\, \subset\, {\mathcal F}^{(r-1)}\, \subset\, {\mathcal F}^{(r)}
\end{equation}
of ${\mathcal F}^{(r)}$. Note that ${\mathcal F}^{(r)}$ is the holomorphic vector
bundle underlying the parabolic vector bundle $\text{Sym}^{r-1}(\mathcal{E}_*)$.

\begin{remark}\label{rem2}
It should be clarified that although the holomorphic vector
bundle underlying the parabolic vector bundle $\mathcal{E}_*$, namely the holomorphic vector
bundle $\mathcal{E}$, does not depend on the function $\mathbf{c}$ in \eqref{e37}, the vector
bundle ${\mathcal F}^{(r)}$ depends on $\mathbf{c}$ in general. It
should also be mentioned that the fact that the parabolic exterior product
$\bigwedge^r\text{Sym}^{r-1}(\mathcal{E}_*)$ is the trivial holomorphic line bundle
equipped with the trivial parabolic structure does not imply that
$\bigwedge^r{\mathcal F}^{(r)}$ is the trivial holomorphic line bundle. In fact
we have $\text{degree}\left(\bigwedge^r{\mathcal F}^{(r)}\right)\, <\, 0$. Note that
$\text{degree}(\mathcal{E})\,=\, -n$.
\end{remark}

\begin{remark}\label{rem3}
It is a straightforward computation to check that
$$\bigwedge\nolimits^r{\mathcal F}^{(r)}\,=\, {\mathcal O}_X(-\sum_{i=1}^n d_ix_i)\, ,$$
where
$$d_i\,=\, \sum_{k=0}^{r-1}\left(\frac{2k(c_i-1)+r-1}{2c_i}-
\Big\lfloor{\frac{2k(c_i-1)+r-1}{2c_i}}\Big\rfloor\right)\, ;$$
the integral part of $b\, \in\, \mathbb Q$ is denoted $\lfloor{b}\rfloor$, so
$\lfloor{b}\rfloor\, \in\, \mathbb Z$ and $0\,\leq\, b-\lfloor{b}\rfloor\, <\, 1$. It
is straightforward to check that each $d_i$ is an integer.
\end{remark}

Consider the Atiyah bundle $$\text{At}\left(\text{Sym}^{r-1}(\mathcal{E}_*)\right)\, \subset\,
\text{Diff}^1\left({\mathcal F}^{(r)}, \, {\mathcal F}^{(r)}\right)$$
constructed as in \eqref{e9} for the parabolic vector bundle
$E_*\,=\, \text{Sym}^{r-1}(\mathcal{E}_*)$, where
${\mathcal F}^{(r)}$ is the vector bundle in \eqref{e40} and $\text{Sym}^{r-1}(\mathcal{E}_*)$
is the parabolic vector bundle in \eqref{e39}. Let
\begin{equation}\label{e41b}
\text{At}'_X(r)\, \subset\, \text{At}\left(\text{Sym}^{r-1}(\mathcal{E}_*)\right)\, \subset\,
\text{Diff}^1\left({\mathcal F}^{(r)}, \, {\mathcal F}^{(r)}\right)
\end{equation}
be the holomorphic subbundle of $\text{At}\left(\text{Sym}^{r-1}(\mathcal{E}_*)\right)$ constructed as follows:
The space of all holomorphic sections of $\text{At}'_X(r)$ over any open subset $U\, \subset \, X$
is the space of all first order differential operators
$$D_U\, :\, \Gamma\left(U,\, \text{Sym}^{r-1}(\mathcal{E}_*)\right)\, \longrightarrow\,
\Gamma\left(U,\, \text{Sym}^{r-1}(\mathcal{E}_*)\right)$$
such that
\begin{itemize}
\item $D_U\, \in\, \Gamma\left(U,\, \text{At}\left(\text{Sym}^{r-1}(\mathcal{E}_*)
\right)\right)$, and

\item $D_U\left({\mathcal F}^{(j)}\right)\, \subset\, {\mathcal F}^{(j)}$ for all $1\,\leq\, j\,
\leq\, r$, where ${\mathcal F}^{(j)}$ is the subbundle in \eqref{e40}.
\end{itemize}

Consider the subbundle $${\mathcal O}_X\, \subset\,
\text{Hom}\left({\mathcal F}^{(r)}, \,{\mathcal F}^{(r)}\right)\,=\,
\text{Diff}^0\left({\mathcal F}^{(r)}, \, {\mathcal F}^{(r)}\right)
\, \subset\, \text{Diff}^1\left({\mathcal F}^{(r)}, \, {\mathcal F}^{(r)}\right)$$ given by
pointwise multiplication. We note that
$$
{\mathcal O}_X\, \subset\, \text{At}'_X(r)\, ,
$$
where $\text{At}'_X(r)$ is the subsheaf in \eqref{e41b}. Let
\begin{equation}\label{e41}
\text{At}_X(r)\, :=\, \text{At}'_X(r)/{\mathcal O}_X
\end{equation}
be the quotient bundle.

Construct the holomorphic vector bundle 
$\text{End}^n\left(\text{Sym}^{r-1}(\mathcal{E}_*)\right)$ by substituting, in 
\eqref{e5}, the parabolic vector bundle $\text{Sym}^{r-1}(\mathcal{E}_*)$ in place of 
$E_*$. Let
$$
\text{ad}^n\left(\text{Sym}^{r-1}(\mathcal{E}_*)\right)\, \subset\,
\text{End}^n\left(\text{Sym}^{r-1}(\mathcal{E}_*)\right)
$$
be the holomorphic subbundle of co-rank one given by the intersection
of $\text{End}^n\left(\text{Sym}^{r-1}(\mathcal{E}_*)\right)$ with the
sheaf of endomorphisms of ${\mathcal F}^{(r)}$ of trace zero.
So using the natural inclusion map ${\mathcal O}_X(-S)\,\hookrightarrow\,
\text{End}^n\left(\text{Sym}^{r-1}(\mathcal{E}_*)\right)$ defined by
pointwise multiplication, we have
\begin{equation}\label{e41c}
\text{End}^n\left(\text{Sym}^{r-1}(\mathcal{E}_*)\right)\,=\,
\text{ad}^n\left(\text{Sym}^{r-1}(\mathcal{E}_*)\right)
\oplus {\mathcal O}_X(-S)\, .
\end{equation}
Let
$$
\text{End}^n_1\left(\text{Sym}^{r-1}(\mathcal{E}_*)\right)\, \subset\,
\text{End}^n\left(\text{Sym}^{r-1}(\mathcal{E}_*)\right)
$$
be the subbundle defined by imposing the condition that the subbundle ${\mathcal F}^{(i)}$ in 
\eqref{e40} is mapped to ${\mathcal F}^{(i+1)}$ for all $1\,\leq\, i\, \leq\, r-1$. In 
other words, a locally defined holomorphic section $s$ of the vector bundle
$\text{End}^n(\text{Sym}^{r-1}(\mathcal{E}_*))$ is a locally defined section of 
$\text{End}^n_1\left(\text{Sym}^{r-1}(\mathcal{E}_*)\right)$ if and only if $s\left({\mathcal 
F}^{(i)}\right)\, \subset\, {\mathcal F}^{(i+1)}$ for every $1\,\leq\, i\, \leq\, r-1$. Define
the intersection
\begin{equation}\label{g1}
\text{ad}^n_1\left(\text{Sym}^{r-1}(\mathcal{E}_*)\right)\,:=\,
\text{ad}^n\left(\text{Sym}^{r-1}(\mathcal{E}_*)\right)
\bigcap \text{End}^n_1\left(\text{Sym}^{r-1}(\mathcal{E}_*)\right)\, ;
\end{equation}
this intersection is taking place inside $\text{End}^n
\left(\text{Sym}^{r-1}(\mathcal{E}_*)\right)$. From \eqref{e41c} we have
\begin{equation}\label{g1b}
\text{End}^n_1\left(\text{Sym}^{r-1}(\mathcal{E}_*)\right)\,=\,
\text{ad}^n_1\left(\text{Sym}^{r-1}(\mathcal{E}_*)\right)
\oplus {\mathcal O}_X(-S)\, .
\end{equation}

Now let $D$ be a connection on the parabolic vector bundle
$\text{Sym}^{r-1}(\mathcal{E}_*)$ defining a parabolic
${\rm SL}(r,{\mathbb C})$--oper on $X$. In other words,
the connection on $\bigwedge^r\text{Sym}^{r-1}(\mathcal{E}_*)$
induced by $D$ coincides with the connection on ${\mathcal O}_X$ given by the de
Rham differential $d$. Consider the first order holomorphic differential operator
\begin{equation}\label{k6}
{\mathcal D}\, :\, \text{At}\left(\text{Sym}^{r-1}(\mathcal{E}_*)\right)\, \longrightarrow\,
\text{End}^n\left(\text{Sym}^{r-1}(\mathcal{E}_*)\right)\otimes K_X\otimes{\mathcal O}_X(S)
\end{equation}
constructed as in \eqref{e21} from $D$; substitute $(\text{Sym}^{r-1}(\mathcal{E}_*),\, D)$
in place of $(E_*,\, D)$ in \eqref{e21}.
The restriction of the differential operator ${\mathcal D}$ to the subbundle
$\text{At}'_X(r)$ constructed in \eqref{e41b} has its image contained in
$$
\text{End}^n_1\left(\text{Sym}^{r-1}(\mathcal{E}_*)\right)\otimes K_X\otimes{\mathcal O}_X(S)\, \subset\,
\text{End}^n\left(\text{Sym}^{r-1}(\mathcal{E}_*)\right)\otimes K_X\otimes{\mathcal O}_X(S)
$$
(see \eqref{g1b}), in other words,
\begin{equation}\label{k7}
{\mathcal D}(\text{At}'_X(r))\, \subset\,
\text{End}^n_1\left(\text{Sym}^{r-1}(\mathcal{E}_*)\right)\otimes K_X\otimes{\mathcal O}_X(S)\, .
\end{equation}
Furthermore, the differential operator ${\mathcal D}$
takes the subbundle ${\mathcal O}_X\, \subset\, \text{At}'_X(r)$ (see \eqref{e41})
to $$K_X\,=\, {\mathcal O}_X(-S)\otimes K_X\otimes{\mathcal O}_X(S)\, \subset\,
\text{End}^n_1\left(\text{Sym}^{r-1}(\mathcal{E}_*)\right)
\otimes K_X\otimes{\mathcal O}_X(S)\, ;$$
see \eqref{g1b} for the subbundle ${\mathcal O}_X(-S)\, \subset\,
\text{End}^n_1\left(\text{Sym}^{r-1}(\mathcal{E}_*)\right)$. In fact, the restriction
of ${\mathcal D}$ to ${\mathcal O}_X\, \subset\, \text{At}'_X(r)$ coincides with the
de Rham differential $d$. Consequently, using
\eqref{k7} and the decomposition in \eqref{g1b}, the differential operator
${\mathcal D}\big\vert_{\text{At}'_X(r)}$ in \eqref{k6} produces a differential operator
\begin{equation}\label{e42}
{\mathcal D}_{\mathcal B}\, :\, \text{At}_X(r)\,
\longrightarrow\,\text{ad}^n_1\left(\text{Sym}^{r-1}(\mathcal{E}_*)\right)\otimes
K_X\otimes{\mathcal O}_X(S)\, .
\end{equation}

Let ${\mathcal C}_\bullet$ denote the following two-term complex of sheaves on $X$
\begin{equation}\label{e43}
{\mathcal C}_\bullet\,\,:\,\, {\mathcal C}_0\,=\, \text{At}_X(r)\,
\stackrel{\mathcal{D}_{\mathcal B}}{\longrightarrow}\, {\mathcal C}_1\,=\,\text{ad}^n_1
\left(\text{Sym}^{r-1}(\mathcal{E}_*)\right)\otimes K_X\otimes{\mathcal O}_X(S)\, ,
\end{equation}
where $\mathcal{D}_{\mathcal B}$ is the homomorphism in \eqref{e42}, and
${\mathcal C}_i$ is at the $i$-th position.

\begin{theorem}\label{thm1}
The space of all infinitesimal deformations of the triple $(X,\, S,\, D)$, where $D$ is
a parabolic ${\rm SL}(r,{\mathbb C})$--oper on $X$, is given by the hypercohomology group
$$
{\mathbb H}^1({\mathcal C}_\bullet),\,
$$
where ${\mathcal C}_\bullet$ is the complex in \eqref{e43}.
\end{theorem}

\begin{proof}
As in \eqref{gc}, take a ramified Galois covering
\begin{equation}\label{gc2}
\varphi\, :\, Y\, \longrightarrow\, X
\end{equation}
satisfying the following two conditions:
\begin{itemize}
\item $\varphi$ is unramified over the complement $X\setminus S$, and

\item for every $x_i\, \in\, S$ and each point $y\, \in\, \varphi^{-1}(x_i)$, the order of
ramification of $\rho$ at $y$ is $c_i\,=\, \mathbf{c}(x_i)$ (see \eqref{e37}).
\end{itemize}
Such a covering $\varphi$ exists by \cite[p. 26, Proposition 1.2.12]{Na}. As before,
denote by $\Gamma_\varphi$ the Galois group $\text{Gal}(\varphi)$ of the map
$\varphi$. From \cite[p.~514, Theorem 6.3]{BDP} we know that the
parabolic ${\rm SL}(r,{\mathbb C})$--oper $D$ on $X$ corresponds to a $\Gamma$--invariant
${\rm PSL}(r,{\mathbb C})$--oper $\mathbb D$ on $Y$. We note that there is a natural bijection
between the ${\rm PSL}(r,{\mathbb C})$--opers on $Y$ and each connected component of the
${\rm SL}(r,{\mathbb C})$--opers on $Y$.
We will recall a description of this ${\rm PSL}(r,{\mathbb C})$--oper $\mathbb D$ on $Y$
corresponding to $D$.

We first note that the parabolic vector bundle $\text{Sym}^{r-1}(\mathcal{E}_*)$ defines a 
holomorphic parabolic principal $\text{SL}(r,{\mathbb C})$--bundle on $X$, because
$\bigwedge^r \text{Sym}^{r-1}(\mathcal{E}_*)$ is the trivial
parabolic line bundle; see 
\cite{BBN1}, \cite{BBN2}, \cite{BBP} for the parabolic analog of principal bundles. Using 
the quotient map $\text{SL}(r,{\mathbb C})\, \longrightarrow\, \text{PSL}(r,{\mathbb C})$, 
a parabolic principal $\text{SL}(r,{\mathbb C})$--bundle produces parabolic principal 
$\text{PSL}(r,{\mathbb C})$--bundle. Let
\begin{equation}\label{ec3}
\widetilde{\mathbb P}_*\, \longrightarrow\, X
\end{equation}
denote the parabolic principal $\text{PSL}(r,{\mathbb C})$--bundle
given by the parabolic principal 
$\text{SL}(r,{\mathbb C})$--bundle on $X$ defined by $\text{Sym}^{r-1}(\mathcal{E}_*)$.

Fix a Borel subgroup
\begin{equation}\label{ec4}
B\, \subset\, \text{PSL}(r,{\mathbb C})\, .
\end{equation}
The filtration of parabolic vector bundles $\{{\mathcal F}^{(j)}_*\}_{j=1}^{r-1}$
in \eqref{e39} produces a reduction of structure group of the
parabolic principal $\text{PSL}(r,{\mathbb C})$--bundle $\widetilde{\mathbb P}_*$
in \eqref{ec3} to the subgroup $B$ of $\text{PSL}(r,{\mathbb C})$ in \eqref{ec4}. Let
\begin{equation}\label{ec5}
\widetilde{\mathbb P}(B)_*\, \subset\, \widetilde{\mathbb P}_*
\end{equation}
be this reduction of structure group to $B$.

The parabolic principal $\text{PSL}(r,{\mathbb C})$--bundle $\widetilde{\mathbb P}_*$ 
in \eqref{ec3}
corresponds to an equivariant holomorphic principal $\text{PSL}(r,{\mathbb C})$--bundle on 
$Y$ \cite{BBN2}, \cite{BBN1}; let
\begin{equation}\label{n2}
{\mathbb P}(r)\, \longrightarrow\, Y
\end{equation}
be the equivariant holomorphic principal 
$\text{PSL}(r,{\mathbb C})$--bundle that corresponds to $\widetilde{\mathbb P}_*$. We note 
that ${\mathbb P}(r)$ is the unique holomorphic principal $\text{PSL}(r,{\mathbb 
C})$--bundle on $Y$ underlying the $\text{PSL}(r,{\mathbb C})$--opers on $Y$. We will 
briefly recall a description of ${\mathbb P}(r)$.

Take a theta characteristic $K^{1/2}_Y$ on $Y$. Let
\begin{equation}\label{tw}
0\, \longrightarrow\, K^{1/2}_Y \, \longrightarrow\, W \, \longrightarrow\,
\left(K^{1/2}_Y\right)^* \, \longrightarrow\, 0
\end{equation}
be the nontrivial extension corresponding to
$$1\, \in\, {\mathbb C}\,=\, H^0(Y,\,{\mathcal O}_Y)^*\,=\, H^1(Y,\, K_Y)\,=\, H^1(Y,\, 
\text{Hom}((K^{1/2}_Y)^*,\, K^{1/2}_Y))\, .$$ The holomorphic principal 
$\text{PSL}(r,{\mathbb C})$--bundle ${\mathbb P}(r)$
in \eqref{n2} coincides with the holomorphic 
principal $\text{PSL}(r,{\mathbb C})$--bundle defined by the projective bundle ${\mathbb 
P}\left(\text{Sym}^{r-1}(W)\right)$. (By $\mathbb{P}(V)$ we denote the
projective bundle defined by the spaces of lines in the fibers of $V$.)
While the vector bundle $\text{Sym}^{r-1}(W)$ depends on the 
choice of the theta characteristic $K^{1/2}_Y$ when $r$ is even (the vector bundle 
$\text{Sym}^{r-1}(W)$ does not depend on the choice of $K^{1/2}_Y$ when $r$ is odd), the 
projective bundle ${\mathbb P}\left(\text{Sym}^{r-1}(W)\right)$ is actually independent of the choice 
of the theta characteristic $K^{1/2}_Y$. Indeed, if we replace $K^{1/2}_Y$ by 
$K^{1/2}_Y\otimes \xi$, where $\xi$ is a holomorphic line bundle on $Y$ of order two, then
$W$ gets replaced by $W\otimes\xi$, and hence
$\text{Sym}^{r-1}(W)$ gets replaced by $\text{Sym}^{r-1}(W)\otimes \xi^{\otimes (r-1)}$.

Since any holomorphic automorphism of $Y$ takes a theta characteristic on $Y$ to a 
(possibly different) theta characteristic on $Y$, from the above property of ${\mathbb 
P}\left(\text{Sym}^{r-1}(W)\right)$ it follows immediately that ${\mathbb P}(r)$ is an 
equivariant holomorphic principal $\text{PSL}(r,{\mathbb C})$--bundle on $Y$. In other 
words, the action of $\Gamma_\varphi$ on $Y$ lifts to a holomorphic action of 
$\Gamma_\varphi$ on ${\mathbb P}(r)$ that commutes with the action of 
$\text{PSL}(r,{\mathbb C})$ on the principal bundle ${\mathbb P}(r)$.

In particular, ${\mathbb P}(W)$ is an equivariant vector bundle. The action of
$\Gamma_\varphi$ on ${\mathbb P}(W)$ preserves the holomorphic section of the projective
bundle ${\mathbb P}(W)$ given by the line subbundle $K^{1/2}_Y$ in \eqref{tw}.

The connection $D$ on the parabolic vector bundle $\text{Sym}^{r-1}(\mathcal{E}_*)$ produces
a connection on the parabolic principal $\text{PSL}(r,{\mathbb C})$--bundle
$\widetilde{\mathbb P}_*$ on $X$ given by ${\mathbb P}\left(\text{Sym}^{r-1}(\mathcal{E}_*)\right)$. This
connection on $\widetilde{\mathbb P}_*$ in turn 
produces a $\Gamma_\varphi$--invariant holomorphic connection on the
principal $\text{PSL}(r,{\mathbb C})$--bundle ${\mathbb P}(r)$ on $Y$. This
$\Gamma_\varphi$--invariant holomorphic connection on ${\mathbb P}(r)$ will
be denoted by
\begin{equation}\label{g2}
D^Y.
\end{equation}

Any holomorphic connection on ${\mathbb P}(r)$ is a $\text{PSL}(r,{\mathbb C})$--oper on 
$Y$ \cite{BD1}, \cite{BD2}, \cite{Fr}. In particular, $D^Y$ in \eqref{g2} is a 
$\text{PSL}(r,{\mathbb C})$--oper on $Y$. Given an oper, a complex of sheaves was 
constructed by Sanders in \cite{Sa}. We will briefly recall his construction of complex
of sheaves for the $\text{PSL}(r,{\mathbb C})$--oper $D^Y$ in \eqref{g2}.

Let
$$
{\mathcal W}_1\, \subset\, {\mathcal W}_2 \, \subset\, \cdots \, \subset\,
{\mathcal W}_{r-1} \, \subset\,{\mathcal W}_r\,=\, \text{Sym}^{r-1}(W)
$$
be the filtration of holomorphic subbundles, where ${\mathcal W}_j\,=\,
(K^{1/2}_Y)^{\otimes (r-j)}\otimes \text{Sym}^{j-1}(W)$ (see \eqref{tw});
in particular ${\rm rank}({\mathcal W}_j)\,=\, j$. Let
\begin{equation}\label{n1}
{\mathbb P}({\mathcal W}_1)\, \subset\, {\mathbb P}({\mathcal W}_2) \, \subset\, \cdots \,
\subset\,{\mathbb P}({\mathcal W}_{r-1}) \, \subset\,{\mathbb P}({\mathcal W}_r)\,=\,
{\mathbb P}(\text{Sym}^{r-1}(W))
\end{equation}
be the corresponding filtration of projective bundles; recall that $\mathbb{P}(V)$ denotes 
the projective bundle defined by the spaces of lines in the fibers of $V$. The filtration 
in \eqref{n1} produces a holomorphic reduction of structure group
\begin{equation}\label{br}
{\mathbb P}(r)_B\, \subset\, {\mathbb P}(r)
\end{equation}
to the Borel subgroup $B$ in \eqref{ec4}, where ${\mathbb P}(r)$ is the
holomorphic principal $\text{PSL}(r,{\mathbb C})$--bundle in \eqref{n2}.

Recall that ${\mathbb P}(W)$ is an equivariant vector bundle, and the action of 
$\Gamma_\varphi$ on ${\mathbb P}(W)$ preserves the holomorphic section of ${\mathbb P}(W)$ 
given by the line subbundle $K^{1/2}_Y$ in \eqref{tw}. Therefore, from the construction of 
the filtration in \eqref{n1} it follows immediately that the action of $\Gamma_\varphi$ on 
${\mathbb P}(\text{Sym}^{r-1}(W))$ preserves each projective subbundle ${\mathbb 
P}({\mathcal W}_j)$. Consequently, the action of $\Gamma_\varphi$ on ${\mathbb P}(r)$ 
preserves the reduction of the structure group ${\mathbb P}(r)_B$ in \eqref{br}. The 
principal $B$--bundle ${\mathbb P}(r)_B$ in \eqref{br} in fact corresponds to the 
reduction $\widetilde{\mathbb P}(B)_*$ in \eqref{ec5}.

The holomorphic reduction ${\mathbb P}(r)_B$ in \eqref{br} coincides 
with the holomorphic reduction of structure group of ${\mathbb P}(r)$ to the subgroup $B\, 
\subset\, \text{PSL}(r,{\mathbb C})$ that appears in the definition of a 
$\text{PSL}(r,{\mathbb C})$--oper on $Y$. Let
$$
\text{At}({\mathbb P}(r)_B)\, \longrightarrow\, Y
$$
be the Atiyah bundle for the holomorphic principal $B$--bundle
${\mathbb P}(r)_B$ (see \cite{At}). Let
$$
\text{At}({\mathbb P}(r))\, \longrightarrow\, Y
$$
be the Atiyah bundle for the principal $\text{PSL}(r,{\mathbb C})$--bundle ${\mathbb P}(r)$
in \eqref{n2}. We have
\begin{equation}\label{e44}
\text{At}({\mathbb P}(r)_B)\, \subset\, \text{At}({\mathbb P}(r))
\end{equation}
because of the reduction of structure group in \eqref{br}.

Let $\text{ad}({\mathbb P}(r))\, \longrightarrow\, Y$ be the adjoint bundle
of the holomorphic principal $\text{PSL}(r,{\mathbb C})$--bundle ${\mathbb P}(r)$. We recall
that $\text{ad}({\mathbb P}(r))$ is the holomorphic vector bundle associated to
the principal $\text{PSL}(r,{\mathbb C})$--bundle ${\mathbb P}(r)$ for the
adjoint action of $\text{PSL}(r,{\mathbb C})$ on its Lie algebra. We will describe
$\text{ad}({\mathbb P}(r))$ explicitly. Let
$$
\gamma\, :\, {\mathbb P}(\text{Sym}^{r-1}(W))\, \longrightarrow\, Y
$$
be the natural projection. Let
$$
T_\gamma\, \subset\, T{\mathbb P}(\text{Sym}^{r-1}(W))
$$
be the relative holomorphic tangent bundle for the projection $\gamma$, meaning $T_\gamma$ 
is the kernel of the differential $d\gamma\, :\, T{\mathbb P}(\text{Sym}^{r-1}(W)) \, 
\longrightarrow\, \gamma^* TY$ of $\gamma$. Then we have
\begin{equation}\label{n3}
\text{ad}({\mathbb P}(r))\,=\, \gamma_*T_\gamma\, .
\end{equation}

Given any holomorphic connection on the principal $\text{PSL}(r,{\mathbb C})$--bundle 
${\mathbb P}(r)$, there is a holomorphic differential operator of order one from
$\text{At}({\mathbb P}(r))$ to $\text{ad}({\mathbb P}(r))\otimes K_Y$ \cite[p.~1415, (1)]{Ch2},
\cite[p.~1415, Proposition 4.4]{Ch2}. Consider the
differential operator $$\text{At}({\mathbb P}(r))\, \longrightarrow\, \text{ad}({\mathbb P}(r))
\otimes K_Y$$ corresponding to the connection $D^Y$ in \eqref{g2}. Let
\begin{equation}\label{g3}
\widetilde{D}^Y\, :\, \text{At}({\mathbb P}(r)_B)\, \longrightarrow\,
\text{ad}({\mathbb P}(r))\otimes K_Y
\end{equation}
be its restriction to the subbundle $\text{At}({\mathbb P}(r)_B)$ (see \eqref{e44}).

Let $\text{ad}({\mathbb P}(r)_B)\, \longrightarrow\, Y$ be the adjoint bundle for the
holomorphic principal $B$--bundle ${\mathbb P}(r)_B$. Note that we have
$$
\text{ad}({\mathbb P}(r)_B)\,\subset\, \text{ad}({\mathbb P}(r))
$$
because of the reduction of structure group in \eqref{br}. Recall the description
of $\text{ad}({\mathbb P}(r))$ in \eqref{n3}. For any $y\, \in\, Y$, the subspace
$$
\text{ad}({\mathbb P}(r)_B)_y\,\subset\, \text{ad}({\mathbb P}(r))_y
$$
consists of all holomorphic vector fields $v$ on ${\mathbb P}(\text{Sym}^{r-1}(W))_y$
satisfying the following condition: for any $1\, \leq\, j\, \leq\, r-1$, and
any
$$
z\, \in\, {\mathbb P}({\mathcal W}_j)_y \, \subset\,
{\mathbb P}(\text{Sym}^{r-1}(W))_y
$$
(see \eqref{n1}),
$$
v(z)\, \in\, T_z {\mathbb P}({\mathcal W}_j)_y\, ,
$$
note that $T_z {\mathbb P}({\mathcal W}_j)_y\, \subset\,T_z{\mathbb P}(\text{Sym}^{r-1}(W))_y$
because ${\mathbb P}({\mathcal W}_j)_y \, \subset\,
{\mathbb P}(\text{Sym}^{r-1}(W))_y$. Let
$$
R_n(\text{ad}({\mathbb P}(r)_B))\, \subset\, \text{ad}({\mathbb P}(r)_B)
$$
be the subbundle given by the nilpotent radical bundle; so for any $y\, \in\, Y$,
the subspace $R_n(\text{ad}({\mathbb P}(r)_B))_y\, \subset\, \text{ad}({\mathbb P}(r)_B)_y$
is the nilpotent radical. So, for any $y\, \in\, Y$, we have
$$
R_n(\text{ad}({\mathbb P}(r)_B))_y\,=\, [\text{ad}({\mathbb P}(r)_B)_y,\,
\text{ad}({\mathbb P}(r)_B)_y]\, .
$$
Let
$$
[R_n(\text{ad}({\mathbb P}(r)_B),\, R_n(\text{ad}({\mathbb P}(r)_B)]
\, \subset\,R_n(\text{ad}({\mathbb P}(r)_B)
$$
be the commutator. For any $y\, \in\, Y$, we have
$$
[R_n(\text{ad}({\mathbb P}(r)_B),\, R_n(\text{ad}({\mathbb P}(r)_B)]_y\,=\,
[R_n(\text{ad}({\mathbb P}(r)_B)_y,\, R_n(\text{ad}({\mathbb P}(r)_B)_y]\, .
$$
Let
\begin{equation}\label{e45}
\text{ad}_1({\mathbb P}(r)_B)\, :=\,
[R_n(\text{ad}({\mathbb P}(r)_B),\, R_n(\text{ad}({\mathbb P}(r)_B)]^\perp
\, \subset\, \text{ad}({\mathbb P}(r))
\end{equation}
be the annihilator of $[R_n(\text{ad}({\mathbb P}(r)_B)),\,
R_n(\text{ad}({\mathbb P}(r)_B))]$ for the fiberwise Killing form on
the adjoint bundle $\text{ad}({\mathbb P}(r))$. The vector bundle
$\text{ad}_1({\mathbb P}(r)_B)$ has the following description in terms of the
isomorphism in \eqref{n3}. For any $y\, \in\, Y$, a holomorphic vector field
$v$ on ${\mathbb P}(\text{Sym}^{r-1}(W))_y$ lies in the fiber
$\text{ad}_1({\mathbb P}(r)_B)_y$ if and only
if the following condition holds:
for any $1\, \leq\, j\, \leq\, r-1$ and
any $z\, \in\, {\mathbb P}({\mathcal W}_j)_y$,
$$
v(z)\, \in\, T_z {\mathbb P}({\mathcal W}_{j+1})_y\, .
$$

Since the connection $D^Y$ in \eqref{g2} is a $\text{PSL}(r,{\mathbb C})$--oper
on $Y$, the image
of the differential operator $\widetilde{D}^Y$ in \eqref{g3} is contained in the subbundle
$$
\text{ad}_1({\mathbb P}(r)_B)\otimes K_Y\, \subset\, \text{ad}({\mathbb P}(r))\otimes K_Y
$$
defined in \eqref{e45}. Therefore, $\widetilde{D}^Y$ defines a differential operator
\begin{equation}\label{e46}
{D}^Y_1\, :\, \text{At}({\mathbb P}(r)_B)\, \longrightarrow\,
\text{ad}_1({\mathbb P}(r)_B)\otimes K_Y\, ;
\end{equation}
see \cite[(5.8)]{Sa}. Let ${\mathcal H}_\bullet$ be the following two-term complex of
sheaves on $Y$
\begin{equation}\label{e47}
{\mathcal H}_\bullet\,\,:\,\, {\mathcal H}_0\,=\, \text{At}({\mathbb P}(r)_B)\,
\stackrel{D^Y_1}{\longrightarrow}\, {\mathcal H}_1\,=\,
\text{ad}_1({\mathbb P}(r)_B)\otimes K_Y\, ,
\end{equation}
where $D^Y_1$ is the homomorphism in \eqref{e46}, and
${\mathcal H}_i$ is at the $i$-th position.

The space of all infinitesimal deformations of the $\text{PSL}(r,{\mathbb C})$--oper $(Y,\, D^Y)$
is given by the hypercohomology group ${\mathbb H}^1({\mathcal H}_\bullet)$, where
${\mathcal H}_\bullet$ is the complex constructed in \eqref{e47} \cite[Theorem 5.9]{Sa}.

We noted earlier that the action of $\Gamma_\varphi$ on ${\mathbb P}(r)$ 
preserves the reduction of structure group
${\mathbb P}(r)_B$ in \eqref{br}. Since ${\mathbb P}(r)_B$ is an equivariant 
bundle, we conclude that both $\text{At}({\mathbb P}(r)_B)$ and
$\text{ad}({\mathbb P}(r)_B)$ are equivariant vector bundles. Hence
$R_n(\text{ad}({\mathbb P}(r)_B))$ in \eqref{e45} is an equivariant subbundle of
$\text{ad}({\mathbb P}(r)_B)$,
which in turn implies that $\text{ad}_1({\mathbb P}(r)_B)$ in \eqref{e45} is an equivariant
subbundle of $\text{ad}({\mathbb P}(r))$; the fiberwise Killing form on
$\text{ad}({\mathbb P}(r))$ is evidently $\Gamma_\varphi$--invariant.
The operator ${D}^Y_1$ in \eqref{e46} is $\Gamma_\varphi$--equivariant, because
the connection $D^Y$ in \eqref{g2} is invariant under the action of $\Gamma_\varphi$
on ${\mathbb P}(r)$. Consequently,
the complex ${\mathcal H}_\bullet$ in \eqref{e47} is $\Gamma_\varphi$--equivariant.
Therefore, the group $\Gamma_\varphi$ acts on
the hypercohomology group ${\mathbb H}^1({\mathcal H}_\bullet)$. Let
\begin{equation}\label{e48}
{\mathbb H}^1({\mathcal H}_\bullet)^{\Gamma_\varphi}\, \subset\,
{\mathbb H}^1({\mathcal H}_\bullet)
\end{equation}
be the invariant part for the action of $\Gamma_\varphi$.

We have 
$$
(\varphi_*\text{At}({\mathbb P}(r)_B))^{\Gamma_\varphi}\,=\, \text{At}_X(r)
$$
and
$$
(\varphi_*(\text{ad}_1({\mathbb P}(r)_B)\otimes K_Y))^{\Gamma_\varphi}
\,=\, \text{ad}^n_1\left(\text{Sym}^{r-1}(\mathcal{E}_*)\right)\otimes K_X\otimes{\mathcal O}_X(S)
$$ 
(see \eqref{e43} for $\text{At}_X(r)$ and $\text{ad}^n_1
(\text{Sym}^{r-1}(\mathcal{E}_*))\otimes K_X\otimes{\mathcal O}_X(S)$). Moreover, the differential
operator ${D}^Y_1$ in \eqref{e46} gives the differential operator $\mathcal{D}_B$ in
\eqref{e43}. Therefore, we conclude that
$$
{\mathbb H}^1({\mathcal C}_\bullet)\,=\, {\mathbb H}^1({\mathcal H}_\bullet)^{\Gamma_\varphi}\, ,
$$
where ${\mathcal C}_\bullet$ is the complex in \eqref{e43} (see \eqref{e48}). Now the theorem
follows from the earlier mentioned result of \cite{Sa} that
the infinitesimal deformations of the $\text{PSL}(r,{\mathbb C})$--oper $(Y,\, D^Y)$
are parametrized by ${\mathbb H}^1({\mathcal H}_\bullet)$.
\end{proof}

Consider the short exact sequence of holomorphic vector bundles on $X$
\begin{equation}\label{k1}
0\, \longrightarrow\, {\mathcal O}_X\, \longrightarrow\,\text{At}'_X(r)
\, \longrightarrow\, \text{At}_X(r)\, \longrightarrow\, 0
\end{equation}
in \eqref{e41}. We will show that it splits holomorphically.

Take any
$$\delta\, \in\, \Gamma\left(U,\, \text{At}\left(\text{Sym}^{r-1}(\mathcal{E}_*)\right)\right)
\, \subset\, \Gamma\left(U,\,
\text{Diff}^1\left({\mathcal F}^{(r)}, \, {\mathcal F}^{(r)}\right)\right)
$$
(see \eqref{e41b}), where $U\, \subset\, X$ is an open
subset. Then $\delta$ produces a holomorphic differential operator
\begin{equation}\label{k3}
\widetilde{\delta}\, \in\, \Gamma\left(U,\, \text{Diff}^1
\left(\bigwedge\nolimits^r{\mathcal F}^{(r)}, \, \bigwedge\nolimits^r{\mathcal F}^{(r)}
\right)\right)
\end{equation}
which is constructed as follows: Take any
\begin{equation}\label{k2}
s\, =\, s_1\wedge \cdots\wedge s_r\, \in\, \Gamma\left(U,\,
\bigwedge\nolimits^r{\mathcal F}^{(r)}\right)
\end{equation}
where $s_i\, \in\, \Gamma\left(U,\, {\mathcal F}^{(r)}\right)$ for all
$1\, \leq\, i\, \leq\, r$. Now define
$$
\widetilde{\delta}(s)\,:=\, \sum_{j=1}^r s_1\wedge\cdots\wedge s_{j-1}\wedge
\delta (s_j)\wedge s_{j+1}\wedge \cdots\wedge s_r\, \in\, \Gamma\left(U,\,
\bigwedge\nolimits^r{\mathcal F}^{(r)}\right)\, .
$$
It is straightforward to check that $\widetilde{\delta}(s)$ is indeed independent of
the choice of the decomposition of the section $s$ in \eqref{k2}.

Let
\begin{equation}\label{k4}
\eta\,\,:\,\, \text{At}'_X(r)\, \longrightarrow\, TX\otimes{\mathcal O}_X(-S)
\end{equation}
be the restriction of the natural projection $\text{At}\left(\text{Sym}^{r-1}
(\mathcal{E}_*)\right)\, \longrightarrow\, TX\otimes{\mathcal O}_X(-S)$
(see \eqref{e10}) to the subbundle $\text{At}'_X(r)\, \subset\,
\text{At}\left(\text{Sym}^{r-1}(\mathcal{E}_*)\right)$ in \eqref{e41b}.

We recall from Remark \ref{rem3} that $\bigwedge\nolimits^r{\mathcal F}^{(r)}
\,=\, {\mathcal O}_X(-\sum_{i=1}^n d_ix_i)$. Therefore, the de Rham differential $d$
on ${\mathcal O}_X$ produces a logarithmic connection on
${\mathcal O}_X(-\sum_{i=1}^n d_ix_i)\,=\, \bigwedge\nolimits^r{\mathcal F}^{(r)}$. Let
\begin{equation}\label{etd}
\widetilde{d}\,\, :\,\, \bigwedge\nolimits^r{\mathcal F}^{(r)} \,\longrightarrow\,
\bigwedge\nolimits^r{\mathcal F}^{(r)} \otimes K_X\otimes
{\mathcal O}_X(\sum_{i=1}^n x_i)
\end{equation}
be the logarithmic connection on $\bigwedge\nolimits^r{\mathcal F}^{(r)}$
given by the de Rham differential.

Let
\begin{equation}\label{k5}
\text{At}^0_X(r)\, \subset\, \text{At}'_X(r)
\end{equation}
be the holomorphic subbundle whose holomorphic sections over any open subset
$U\, \subset\, X$ consist of all
$$
\delta\, \in\, \Gamma(U,\, \text{At}'_X(r))
$$
satisfying the following condition:
$$
\widetilde{\delta}(s)\,=\, \langle\widetilde{d}(s),\, \eta(\delta)\rangle
$$
for all $s\, \in\, \Gamma\left(U,\, \bigwedge\nolimits^r{\mathcal F}^{(r)}\right)$,
where $\widetilde{\delta}$ is constructed in \eqref{k3} from
$\delta$ and $\eta$ is the homomorphism in
\eqref{k4}, while and $\langle-,\, -\rangle$ is the duality pairing in \eqref{dp}
and $\widetilde{d}$ is constructed in \eqref{etd}.

The composition of homomorphisms
$$
\text{At}^0_X(r)\, \hookrightarrow\, \text{At}'_X(r) \, \longrightarrow\, \text{At}_X(r)
$$
(see \eqref{k5} and \eqref{k1} for these homomorphisms) is evidently an
isomorphism. Therefore, the holomorphic subbundle $\text{At}^0_X(r)$ in \eqref{k5} produces
a holomorphic splitting of the short exact sequence in \eqref{k1}.

Let $D$ be a connection on $\text{Sym}^{r-1}(\mathcal{E}_*)$ defining a parabolic
${\rm SL}(r,{\mathbb C})$--oper on $X$.
Consider the differential operator ${\mathcal D}$ in \eqref{k6} constructed
from the ${\rm SL}(r,{\mathbb C})$--oper $D$. Clearly, we have
$$
{\mathcal D}\left(\text{At}^0_X(r)\right)\, \subset\, \text{ad}^n_1
\left(\text{Sym}^{r-1}(\mathcal{E}_*)\right)\otimes K_X\otimes{\mathcal O}_X(S)\, ,
$$
where $\text{ad}^n_1\left(\text{Sym}^{r-1}(\mathcal{E}_*)\right)$
and $\text{At}^0_X(r)$ are constructed in \eqref{g1} and \eqref{k5}
respectively; recall from \eqref{k7} that
$$
{\mathcal D}\left(\text{At}'_X(r)\right)\, \subset\,
\text{End}^n_1\left(\text{Sym}^{r-1}(\mathcal{E}_*)\right)\otimes K_X\otimes{\mathcal O}_X(S)
$$
and $\text{End}^n_1\left(\text{Sym}^{r-1}(\mathcal{E}_*)\right)$ decomposes as
in \eqref{g1b}.

Consequently, the complex of sheaves ${\mathcal C}_\bullet$ in \eqref{e43}
is equivalent to the following complex ${\mathcal C}'_\bullet$ of sheaves on $X$
\begin{equation}\label{k8}
{\mathcal C}'_\bullet\,\,:\,\, {\mathcal C}'_0\,=\, \text{At}^0_X(r)\,
\stackrel{\mathcal{D}}{\longrightarrow}\, {\mathcal C}'_1\,=\,\text{ad}^n_1
\left(\text{Sym}^{r-1}(\mathcal{E}_*)\right)\otimes K_X\otimes{\mathcal O}_X(S)\, .
\end{equation}
Hence Theorem \ref{thm1} gives the following:

\begin{corollary}\label{cor2}
The space of all infinitesimal deformations of the triple $(X,\, S,\, D)$, where $D$ is
a parabolic ${\rm SL}(r,{\mathbb C})$--oper on $X$, is given by the hypercohomology group
$$
{\mathbb H}^1({\mathcal C}'_\bullet),
$$
where ${\mathcal C}'_\bullet$ is the complex in \eqref{k8}.
\end{corollary}

\section{Monodromy of parabolic opers}

Consider a family of $n$--pointed compact Riemann surfaces
\begin{equation}\label{e50}
\{(X_T\, \stackrel{\varpi}{\longrightarrow}\, T),\, (\phi_1,\, \cdots ,\, \phi_n)\}
\end{equation}
as in \eqref{e31} and \eqref{phii}. Assume that this family is locally universal. The
relative canonical line bundle on $X_T$ for the projection $\varpi$ in \eqref{e50}
will be denoted by $K_\varpi$. 

Fix a holomorphic line bundle $\mathbf L$
on $X_T$ such that ${\mathbf L}\otimes{\mathbf L}$ is holomorphically
isomorphic to $K_\varpi\otimes {\mathcal O}_{X_T}(-\sum_{i=1}^n \phi_i(T))$;
such a line bundle $\mathbf L$ exists locally with respect to $T$.
Fix a holomorphic isomorphism between ${\mathbf L}\otimes{\mathbf L}$ and
$K_\varpi\otimes {\mathcal O}_{X_T}(-\sum_{i=1}^n \phi_i(T))$.
Fix a function
$$
\mathbf{c}\, :\, \{1,\, \cdots,\, n\}\, \longrightarrow\,
\{t\, \in\, {\mathbb Z}\, \mid\, t\, \geq\, 2\}
$$
as in \eqref{e37}.

Fix the parabolic structure to be that of a parabolic ${\rm SL}(r,{\mathbb C})$--oper.
As in \eqref{e33},
\begin{equation}\label{e51}
\beta\,\, :\, \,{\mathcal M}^T(r)\, \longrightarrow\, T
\end{equation}
is the corresponding relative moduli space of
parabolic vector bundles with connection. 
Let
\begin{equation}\label{e52}
{\mathcal O}^T(r)\ \hookrightarrow\ {\mathcal M}^T(r)
\end{equation}
be the locus of parabolic ${\rm SL}(r,{\mathbb C})$--opers.

Consider the monodromy map
$$
\mathbb{M}\ :\ {\mathcal M}^T(r)\ \longrightarrow\ {\mathcal R}_X(r)
$$
in \eqref{e32}. Let
\begin{equation}\label{e53}
\mathbb{M}^0\ :\ {\mathcal O}^T(r)\ \longrightarrow\ {\mathcal R}_X(r)
\end{equation}
be the restriction of $\mathbb M$ to the subspace ${\mathcal O}^T(r)$
in \eqref{e52}.

We will prove that the holomorphic map $\mathbb{M}^0$ in \eqref{e53} is an immersion.
Our proof is modeled on the proof of Sanders that
$\mathbb{M}^0$ is an immersion under the assumption that $n\,=\, 0$
(parabolic points are absent); see \cite[Theorem 6.3]{Sa}.

Take a point $t_0\, \in\, T$. Denote the compact Riemann surface $\varpi^{-1}(t_0)\,=\, X_{t_0}$ by $X$.
Denote the divisor $\sum_{i=1}^n \phi_i(t_0)$ on $X$ by $S$.
We have $$T_{t_0}T \,=\, H^1(X, \, TX\otimes{\mathcal O}_X(-S));$$ recall that
$(\varpi,\, \{\phi_i\}_{i=1}^n)$ is a locally complete family. Take any
parabolic vector bundle with connection
\begin{equation}\label{eb}
(E_*,\, D)\ \in\ \beta^{-1}(t_0)
\end{equation}
on $X$, where $\beta$ is the projection in \eqref{e51}. From Lemma \ref{lem4} it follows that
\begin{equation}\label{e55}
T_{(E_*,D)}{\mathcal M}^T(r)\,=\, {\mathbb H}^1({\mathcal B}_\bullet),
\end{equation}
where ${\mathcal B}_\bullet$ is the complex in \eqref{e27}. Let
$$
T_\beta \, \subset\, T{\mathcal M}^T(r)
$$
be the relative tangent bundle for the projection $\beta$ in \eqref{e51}. From
Lemma \ref{lem3} it follows that $(T_\beta)_{(E,D)}\,=\, {\mathbb H}^1({\mathcal A}_\bullet)$,
where ${\mathcal A}_\bullet$ is the complex in \eqref{e25}.
Consider the homomorphism
$$
\delta\,:\, T_{t_0}T \,=\, H^1(X, \, TX\otimes{\mathcal O}_X(-S))\,
\longrightarrow\, T{\mathcal M}^T(r)
\,=\, {\mathbb H}^1({\mathcal B}_\bullet)
$$
in \eqref{e36}. In Lemma \ref{lem5} it was shown that $\delta$ coincides with the homomorphism
of tangent spaces given by the isomonodromy map.
Recall that the two holomorphic foliations
on ${\mathcal M}^T(r)$, one given by the isomonodromy condition and the other given by the
projection $\beta$ in \eqref{e51}, are transversal. Therefore, from Lemma \ref{lem5}
and \eqref{e55} we know that
\begin{equation}\label{e54}
{\mathbb H}^1({\mathcal B}_\bullet)\,=\, (T_\beta)_{(E_*,D)}\oplus \delta(T_{t_0}T)
\,=\, {\mathbb H}^1({\mathcal A}_\bullet)\oplus \delta(H^1(X,\, TX\otimes{\mathcal O}_X(-S))).
\end{equation}
Let
\begin{equation}\label{e56}
\Phi\ :\ {\mathbb H}^1({\mathcal B}_\bullet)\ \longrightarrow\
{\mathbb H}^1({\mathcal A}_\bullet)
\end{equation}
be the projection corresponding to the decomposition of ${\mathbb H}^1({\mathcal B}_\bullet)$ in
\eqref{e54}. We will now describe the homomorphism $\Phi$ in \eqref{e56} explicitly.

The connection $D$ on $E_*$ gives a holomorphic decomposition
\begin{equation}\label{e57}
\text{At}(E_*)\,=\, \text{End}^P(E_*)\oplus (TX \otimes {\mathcal O}_X(-S))
\end{equation}
(see Lemma \ref{lem1}). As in \eqref{e35},
\begin{equation}\label{eb3}
{\mathbf h}\ :\ TX\otimes{\mathcal O}_X(-S)\ \longrightarrow\ {\rm At}(E_*)
\end{equation}
is the homomorphism given by the decomposition in
\eqref{e57}. Consider the homomorphism $\mathcal D \, :\, {\rm At}(E_*)\, \longrightarrow\,
\text{End}^n(E_*) \otimes K_X$ constructed in \eqref{e21}. Recall that
$$\mathcal{D}\circ{\mathbf h}\ =\ 0$$
(see \eqref{e35b}). Consequently, the decomposition in \eqref{e57} produces a homomorphism
$\mathcal P$ of complexes
\begin{equation}\label{e58}
\begin{matrix}
{\mathcal B}_\bullet &: & {\rm At}(E_*) & \stackrel{\mathcal{D}}{\longrightarrow}
& \text{End}^n(E_*)\otimes K_X \otimes{\mathcal O}_X(S)\\
\,\,\,\,\,\Big\downarrow {\mathcal P}&&\,\,\, \Big\downarrow p && \Big\Vert\\
{\mathcal A}_\bullet &: & \text{End}^P(E_*) & \stackrel{\mathcal{D}_0}{\longrightarrow}
& \text{End}^n(E_*)\otimes K_X \otimes{\mathcal O}_X(S)
\end{matrix}
\end{equation}
where $p$ is the projection given by the decomposition in \eqref{e57}. Let
\begin{equation}\label{e59}
{\mathcal P}_*\, \, :\,\,{\mathbb H}^1({\mathcal B}_\bullet)\,\longrightarrow
\, {\mathbb H}^1({\mathcal A}_\bullet)
\end{equation}
be the homomorphism of hypercohomology groups corresponding to the homomorphism of
complexes $\mathcal P$ in \eqref{e58}. The homomorphism ${\mathcal P}_*$
in \eqref{e59} evidently coincides with the projection $\Phi$ in \eqref{e56}. Thus
the homomorphism ${\mathcal P}_*$ furnishes an explicit description of $\Phi$.

Now set $E_*$ in \eqref{eb} to be the rank $r$ parabolic vector bundle $\text{Sym}^{r-1}(\mathcal{E}_*)$
in \eqref{e39}, and let $D$ be a connection on $\text{Sym}^{r-1}(\mathcal{E}_*)$ such that
\begin{equation}\label{eb2}
\left(\text{Sym}^{r-1}(\mathcal{E}_*),\, D\right)\ \in\ {\mathcal O}^T(r)
\bigcap \beta^{-1}(t_0)
\end{equation}
(see \eqref{e52} and \eqref{e51}). From Corollary \ref{cor2} we know that
$$
T_{(\text{Sym}^{r-1}(\mathcal{E}_*),D)}{\mathcal O}^T(r)\,=\, {\mathbb H}^1({\mathcal C}'_\bullet)\, ,
$$
where ${\mathcal C}'_\bullet$ is the complex in \eqref{k8}. We have the following
two homomorphisms $P$ and $Q$ of complexes
\begin{equation}\label{z1}
\begin{matrix}
{\mathcal C}'_\bullet & : & {\mathcal C}'_0\, =\, \text{At}^0_X(r) &
\stackrel{\mathcal{D}}{\longrightarrow}& {\mathcal C}'_1\, =\,\text{ad}^n_1
\left(\text{Sym}^{r-1}(\mathcal{E}_*)\right)\otimes K_X\otimes{\mathcal O}_X(S)\\
\,\,\,\, \Big\downarrow Q &&\,\,\,\,\, \,\Big\downarrow Q_0 &&\,\,\,\,\,\,\Big\downarrow Q_1\\
\widetilde{\mathcal B}_\bullet &:& \widetilde{\mathcal B}_0\,=\,
\text{At}\left(\text{Sym}^{r-1}(\mathcal{E}_*)\right)
& \stackrel{\mathcal{D}}{\longrightarrow}&
\widetilde{\mathcal B}_1\,=\,
\text{End}^n\left(\text{Sym}^{r-1}(\mathcal{E}_*)\right)\otimes K_X\otimes{\mathcal O}_X(S)\\
\,\,\,\, \Big\downarrow P &&\,\,\,\,\, \,\Big\downarrow P_0 &&\,\,\,\,\,\,\Big\downarrow P_1\\
\widetilde{\mathcal A}_\bullet &:&\widetilde{\mathcal A}_0\,=\,{\rm End}^P
\left(\text{Sym}^{r-1}(\mathcal{E}_*)\right) & \stackrel{\mathcal{D}_0}{\longrightarrow}&
\widetilde{\mathcal A}_1\,=\,
\text{End}^n\left(\text{Sym}^{r-1}(\mathcal{E}_*)\right)\otimes K_X\otimes{\mathcal O}_X(S)
\end{matrix}
\end{equation}
where
\begin{itemize}
\item $\widetilde{\mathcal B}_\bullet$ is the complex in \eqref{e27}, with
$\left(\text{Sym}^{r-1}(\mathcal{E}_*),\, D\right)$ substituted in place in $(E_*,\, D)$,

\item $\widetilde{\mathcal A}_\bullet$ is the complex in \eqref{e25}, with
$\left(\text{Sym}^{r-1}(\mathcal{E}_*),\, D\right)$ substituted in place in $(E_*,\, D)$,

\item ${\mathcal C}'_\bullet$ is the complex in \eqref{k8},

\item for $i\,=\, 0,\, 1$, the homomorphism ${\mathcal C}'_i\,
\longrightarrow\, \widetilde{\mathcal B}_i$ in \eqref{z1} is the natural inclusion map, and

\item the homomorphism $P$ is the homomorphism $\mathcal P$ in \eqref{e58}
with $\left(\text{Sym}^{r-1}(\mathcal{E}_*),\, D\right)$ substituted in place in $(E_*,\, D)$.
So $P_1$ is the identity map.
\end{itemize}

Let
\begin{equation}\label{z2}
(P\circ Q)_*\,\ :\,\ {\mathbb H}^1({\mathcal C}'_\bullet)\ \longrightarrow\
{\mathbb H}^1(\widetilde{\mathcal A}_\bullet)
\end{equation}
be the homomorphism of hypercohomology groups induced by the homomorphism of complexes
$P\circ Q$ in \eqref{z1}.

Since the homomorphism ${\mathcal P}_*$
in \eqref{e59} coincides with the projection $\Phi$ in \eqref{e56}, we have the following:

\begin{lemma}\label{lem-t1}
The map $\mathbb{M}^0$ in \eqref{e53} is an immersion if the homomorphism $(P\circ Q)_*$
in \eqref{z2} is injective.
\end{lemma}

We will simplify Lemma \ref{lem-t1} to make it more useful.

The vector bundle $\widetilde{\mathcal B}_0\,=\,
\text{At}\left(\text{Sym}^{r-1}(\mathcal{E}_*)\right)$ in \eqref{z1} has ${\mathcal O}_X$ as a
natural direct summand; in other words, there is a natural decomposition
\begin{equation}\label{t1}
\widetilde{\mathcal B}_0\,=\,
\text{At}\left(\text{Sym}^{r-1}(\mathcal{E}_*)\right)\ =\, {\mathcal O}_X\oplus
\widetilde{\mathcal B}^0_0,
\end{equation}
which will be described below.

Note that the parabolic determinant line bundle $\det \text{Sym}^{r-1}(\mathcal{E}_*)\,=\,
\bigwedge^r \text{Sym}^{r-1}(\mathcal{E}_*)$
is the trivial holomorphic line bundle ${\mathcal O}_X$ with the trivial parabolic structure.
There is a natural homomorphism
$$
\delta_E\ :\ \text{At}\left(\text{Sym}^{r-1}(\mathcal{E}_*)\right) \ \longrightarrow\
\text{At}\left(\det \text{Sym}^{r-1}(\mathcal{E}_*)\right)\ =\ {\mathcal O}_X\oplus (TX\otimes{\mathcal O}_X(-S)) ;
$$
the above decomposition $\text{At}({\mathcal O}_X)\,=\, {\mathcal O}_X\oplus (TX\otimes{\mathcal O}_X(-S))$
is given by the trivial logarithmic connection on the trivial line bundle given by the de Rham
differential $d$ (so this trivial logarithmic connection is given a nonsingular connection). The
above homomorphism $\delta_E$ fits in the following commutative diagram
\begin{equation}\label{t2}
\begin{matrix}
0 & \rightarrow & \text{End}^P(\text{Sym}^{r-1}(\mathcal{E}_*)) & \stackrel{\iota}{\rightarrow} &
{\rm At}(\text{Sym}^{r-1}(\mathcal{E}_*)) & \rightarrow & TX(-S)
& \rightarrow & 0\\
&& \,\,\,\Big\downarrow{\rm tr} && \,\,\, \big\downarrow \delta_E && \Big\Vert\\
0 & \rightarrow & {\mathcal O}_X & \rightarrow &
\text{At}({\mathcal O}_X)\,=\, {\mathcal O}_X\oplus TX(-S) &
\rightarrow & TX(-S) & \rightarrow & 0
\end{matrix}
\end{equation}
where the top (respectively, bottom) exact sequence is the one in \eqref{e10} for the
parabolic vector bundle $\text{Sym}^{r-1}(\mathcal{E}_*)$ (respectively, ${\mathcal O}_X$), and
${\rm tr}$ is the trace map. Note that we have the inclusion map ${\mathcal O}_X\, \hookrightarrow\,
\text{End}^P(\text{Sym}^{r-1}(\mathcal{E}_*))$ that sends any locally defined holomorphic function $f$ to
$f\cdot {\rm Id}_{\text{Sym}^{r-1}(\mathcal{E}_*)}$. So ${\mathcal O}_X$ is a subbundle on
${\rm At}(\text{Sym}^{r-1}(\mathcal{E}_*))$ using the homomorphism $\iota$ in \eqref{t2}.
Define
$$
\widetilde{\mathcal B}^0_0 \ :=\ \delta^{-1}_E(TX\otimes{\mathcal O}_X(-S))\ \subset\
{\rm At}(\text{Sym}^{r-1}(\mathcal{E}_*)).
$$
Now we have the decomposition in \eqref{t1}.

Define $\text{End}^P_0(\text{Sym}^{r-1}(\mathcal{E}_*))\,:=\, \text{kernel}({\rm tr})\,
\subset\, \text{End}^P(\text{Sym}^{r-1}(\mathcal{E}_*))$, where ${\rm tr}$ is the homomorphism
in \eqref{t2}. We have the decomposition 
\begin{equation}\label{t3}
\text{End}^P(\text{Sym}^{r-1}(\mathcal{E}_*)) \ =\ {\mathcal O}_X\oplus
\text{End}^P_0(\text{Sym}^{r-1}(\mathcal{E}_*)).
\end{equation}

Similarly, we have an inclusion map
$$
K_X\ \hookrightarrow\ \text{End}^n\left(\text{Sym}^{r-1}(\mathcal{E}_*)\right)\otimes K_X\otimes{\mathcal O}_X(S)
$$
that sends any locally defined holomorphic $1$-form $\omega$ to ${\rm Id}_{\text{Sym}^{r-1}(\mathcal{E}_*)}
\otimes\omega$, and also have the trace map
$$
{\rm tr}_1\ :\ \text{End}^n\left(\text{Sym}^{r-1}(\mathcal{E}_*)\right)\otimes K_X\otimes{\mathcal O}_X(S)\
\longrightarrow\ K_X.
$$
Define
$$
\text{End}^n_0\left(\text{Sym}^{r-1}(\mathcal{E}_*)\right)\otimes K_X\otimes{\mathcal O}_X(S)\ :=\
\text{kernel}({\rm tr}_1)\ \subset\ \text{End}^n\left(\text{Sym}^{r-1}(\mathcal{E}_*)\right)
\otimes K_X\otimes{\mathcal O}_X(S).
$$
We have the decomposition 
\begin{equation}\label{t4}
\text{End}^n\left(\text{Sym}^{r-1}(\mathcal{E}_*)\right) \otimes K_X\otimes{\mathcal O}_X(S) \ =\ 
K_X\oplus\text{End}^n_0\left(\text{Sym}^{r-1}(\mathcal{E}_*)\right)\otimes
K_X\otimes{\mathcal O}_X(S)
\end{equation}
(see \eqref{e41c}).

Consider the homomorphisms $\mathcal{D}$ and $\mathcal{D}_0$ in \eqref{z1}.
We have
$$
{\mathcal D}(\widetilde{\mathcal B}^0_0)\ \subset\
\text{End}^n_0\left(\text{Sym}^{r-1}(\mathcal{E}_*)\right)\otimes K_X\otimes{\mathcal O}_X(S),
$$
$$
{\mathcal D}_0(\text{End}^P_0(\text{Sym}^{r-1}(\mathcal{E}_*)))\ \subset\
\text{End}^n_0\left(\text{Sym}^{r-1}(\mathcal{E}_*)\right)\otimes K_X\otimes{\mathcal O}_X(S)
$$
(see \eqref{t1}, \eqref{t3} and \eqref{t4}). Therefore, the diagram of homomorphisms in \eqref{z1} gives
the following commutative diagram of homomorphisms:
\begin{equation}\label{t5}
\begin{matrix}
{\mathcal C}'_\bullet & : & {\mathcal C}'_0\, =\, \text{At}^0_X(r) &
\stackrel{\mathcal{D}}{\longrightarrow}& {\mathcal C}'_1\, =\,\text{ad}^n_1
\left(\text{Sym}^{r-1}(\mathcal{E}_*)\right)\otimes K_X\otimes{\mathcal O}_X(S)\\
\,\,\,\, \Big\downarrow Q &&\,\,\,\,\, \,\Big\downarrow Q_0 &&\,\,\,\,\,\,\Big\downarrow Q_1\\
\widetilde{\mathcal B}^0_\bullet &:& \widetilde{\mathcal B}^0_0
& \stackrel{\mathcal{D}}{\longrightarrow}&
\widetilde{\mathcal B}^0_1\,=\,
\text{End}^n_0\left(\text{Sym}^{r-1}(\mathcal{E}_*)\right)\otimes K_X\otimes{\mathcal O}_X(S)\\
\,\,\,\, \Big\downarrow P &&\,\,\,\,\, \,\Big\downarrow P_0 &&\,\,\,\,\,\,\Big\downarrow P_1\\
\widetilde{\mathcal A}^0_\bullet &:&\widetilde{\mathcal A}^0_0 \,=\,{\rm End}^P_0
\left(\text{Sym}^{r-1}(\mathcal{E}_*)\right) & \stackrel{\mathcal{D}_0}{\longrightarrow}&
\widetilde{\mathcal A}^0_1\,=\,
\text{End}^n_0\left(\text{Sym}^{r-1}(\mathcal{E}_*)\right)\otimes K_X\otimes{\mathcal O}_X(S);
\end{matrix}
\end{equation}
the restriction of ${\mathcal D}$ (respectively, ${\mathcal D}_0$) to $\widetilde{\mathcal B}^0_0$
(respectively, $\widetilde{\mathcal A}^0_0$) is also denoted by ${\mathcal D}$
(respectively, ${\mathcal D}_0$), and, similarly, the restriction of $P_0$ (respectively, $P_1$)
is also denoted by $P_0$ (respectively, $P_1$).

Let
\begin{equation}\label{t6}
(P\circ Q)_*\,\ :\,\ {\mathbb H}^1({\mathcal C}'_\bullet)\ \longrightarrow\
{\mathbb H}^1(\widetilde{\mathcal A}^0_\bullet)
\end{equation}
be the homomorphism of hypercohomology groups induced by the homomorphism of complexes
$P\circ Q$ in \eqref{t5}.

The following is an immediate consequence of Lemma \ref{lem-t1}.

\begin{corollary}\label{cor-t1}
The map $\mathbb{M}^0$ in \eqref{e53} is an immersion if the homomorphism $(P\circ Q)_*$
in \eqref{t6} is injective.
\end{corollary}

Consider the homomorphism in \eqref{eb3}
\begin{equation}\label{t7}
{\mathbf h}\ :\ TX\otimes{\mathcal O}_X(-S)\ \longrightarrow\ \widetilde{\mathcal B}^0_0
\end{equation}
(see \eqref{eb2} and \eqref{t1}); we note that the image of ${\mathbf h}$ lies in
$\widetilde{\mathcal B}^0_0\, \subset\, \widetilde{\mathcal B}_0$ because $D$ induces the trivial
parabolic connection on the trivial parabolic bundle $\bigwedge^r \text{Sym}^{r-1}(\mathcal{E}_*)
\,=\, {\mathcal O}_X$. Let
\begin{equation}\label{t8}
\Phi\ :\ TX\otimes{\mathcal O}_X(-S)\ \longrightarrow\ 
\widetilde{\mathcal B}^0_0/Q_0(\text{At}^0_X(r))
\end{equation}
be the homomorphism given by the following composition of homomorphisms:
$$
TX\otimes{\mathcal O}_X(-S)\ \stackrel{\mathbf h}{\longrightarrow}\ \widetilde{\mathcal B}^0_0\
\longrightarrow\ \widetilde{\mathcal B}^0_0/Q_0(\text{At}^0_X(r))
$$
where ${\mathbf h}$ is the homomorphism in \eqref{t7} and $\widetilde{\mathcal B}^0_0\,
\longrightarrow\, \widetilde{\mathcal B}^0_0/Q_0(\text{At}^0_X(r))$ is the natural quotient map.

\begin{proposition}\label{prop-t1}
The map $\mathbb{M}^0$ in \eqref{e53} is an immersion if the homomorphism of cohomologies
$$
\Phi_*\ :\ H^1(X,\, TX\otimes{\mathcal O}_X(-S))\ \longrightarrow\ 
H^1(X,\, \widetilde{\mathcal B}^0_0/Q_0({\rm At}^0_X(r)))
$$
induced by $\Phi$ in \eqref{t8} is injective.
\end{proposition}

\begin{proof}
The criterion in Corollary \ref{cor-t1} will be used. We will first show that the homomorphism of
hypercohomology groups
\begin{equation}\label{t9}
Q_* \ :\ {\mathbb H}^1({\mathcal C}'_\bullet)\ \longrightarrow\ 
{\mathbb H}^1 (\widetilde{\mathcal B}^0_\bullet),
\end{equation}
induced by $Q$ in \eqref{t5}, is injective. To prove this, consider the following
short exact sequence of complexes obtained from \eqref{t5}
\begin{equation}\label{t1a}
\begin{matrix}
&& 0 && 0\\
&& \Big\downarrow && \Big\downarrow\\
{\mathcal C}'_\bullet & : & {\mathcal C}'_0\, =\, \text{At}^0_X(r) &
\stackrel{\mathcal{D}}{\longrightarrow}& {\mathcal C}'_1\\
\,\,\,\, \Big\downarrow Q &&\,\,\,\,\, \,\Big\downarrow Q_0 &&\,\,\,\,\,\,\Big\downarrow Q_1\\
\widetilde{\mathcal B}^0_\bullet &:& \widetilde{\mathcal B}^0_0
& \stackrel{\mathcal{D}}{\longrightarrow}&
\widetilde{\mathcal B}^0_1\\
\Big\downarrow && \Big\downarrow && \Big\downarrow\\
\widetilde{\mathcal A}'_\bullet &:& \widetilde{\mathcal A}'_0\,:=\,
\widetilde{\mathcal B}^0_0/Q_0(\text{At}^0_X(r)) & \longrightarrow &
\widetilde{\mathcal A}'_1\,:=\, \widetilde{\mathcal B}^0_1/Q_1({\mathcal C}'_1)\\
&& \Big\downarrow && \Big\downarrow\\
&& 0 && 0
\end{matrix}
\end{equation}
The short exact sequence of complexes in \eqref{t1a} gives the following long exact
sequence of hypercohomology groups:
\begin{equation}\label{t10}
{\mathbb H}^0(\widetilde{\mathcal A}'_\bullet)\ \longrightarrow\
{\mathbb H}^1({\mathcal C}'_\bullet) \ \longrightarrow\
{\mathbb H}^1(\widetilde{\mathcal B}^0_\bullet).
\end{equation}
On the other hand, we have ${\mathbb H}^0(\widetilde{\mathcal A}'_\bullet)
\, \subset\, H^0(X,\, \widetilde{\mathcal B}^0_0/Q_0(\text{At}^0_X(r)))$.
So using \eqref{t10} it follows that the homomorphism $Q_*$ in \eqref{t9} is
injective if
\begin{equation}\label{t11}
H^0(X,\, \widetilde{\mathcal B}^0_0/Q_0(\text{At}^0_X(r)))\ =\ 0.
\end{equation}

Consider $\text{End}^P_0(\text{Sym}^{r-1}(\mathcal{E}_*))$ in \eqref{t3}. Let
\begin{equation}\label{t12}
\text{End}^P_{\mathcal F}(\text{Sym}^{r-1}(\mathcal{E}_*))\ \subset\
\text{End}^P_0(\text{Sym}^{r-1}(\mathcal{E}_*))
\end{equation}
be the subbundle that preserves the filtration in \eqref{e39}.
Now consider the natural inclusion map
\begin{equation}\label{t12a}
\text{End}^P_0(\text{Sym}^{r-1}(\mathcal{E}_*))
\ \hookrightarrow\ \widetilde{\mathcal B}^0_0
\end{equation}
(see \eqref{e10} and \eqref{t1}). We have
\begin{equation}\label{t13}
\text{End}^P_0(\text{Sym}^{r-1}(\mathcal{E}_*))\cap Q_0(\text{At}^0_X(r))\
=\ \text{End}^P_{\mathcal F}(\text{Sym}^{r-1}(\mathcal{E}_*)),
\end{equation}
where $\text{End}^P_{\mathcal F}(\text{Sym}^{r-1}(\mathcal{E}_*))$ is constructed
in \eqref{t12}.

Using \eqref{t13}, the homomorphism in \eqref{t12a} produces an isomorphism
\begin{equation}\label{t14}
\widetilde{\mathcal B}^0_0/Q_0(\text{At}^0_X(r))\ =\ \text{End}^P_0 (\text{Sym}^{r-1}
(\mathcal{E}_*))/\text{End}^P_{\mathcal F}(\text{Sym}^{r-1}(\mathcal{E}_*)).
\end{equation}

Let
\begin{equation}\label{et}
\mathbb T
\end{equation}
denote the parabolic line bundle on $X$ whose underlying holomorphic
line bundle is $TX\otimes{\mathcal O}_X(-S)$ and the parabolic weight at each $x_i\, \in\,
S$ is $\frac{1}{{\mathbf c}(x_i)}\,=\, \frac{1}{c_i}$ (see \eqref{e37}). The quotient bundle
$$
\text{End}^P_0 (\text{Sym}^{r-1}
(\mathcal{E}_*))/\text{End}^P_{\mathcal F}(\text{Sym}^{r-1}(\mathcal{E}_*))
$$
in \eqref{t14} has a natural filtration of subbundles such that each successive quotient is
of the form ${\mathbb T}^{\otimes k}$ with $k\, \geq\,1$; this filtration is
constructed using the filtration in \eqref{e39}. Using this filtration it follows
immediately that
$$
H^0(X,\, \text{End}^P_0 (\text{Sym}^{r-1}
(\mathcal{E}_*))/\text{End}^P_{\mathcal F}(\text{Sym}^{r-1}(\mathcal{E}_*)))\ =\ 0.
$$
In view of \eqref{t14}, this implies that \eqref{t11} holds. As noted before,
\eqref{t11} implies that the homomorphism $Q_*$ in \eqref{t9} is injective.

Assume that the homomorphism of cohomologies
\begin{equation}\label{t15}
\Phi_*\ :\ H^1(X,\, TX\otimes{\mathcal O}_X(-S))\ \longrightarrow\ 
H^1(X,\, \widetilde{\mathcal B}^0_0/Q_0(\text{At}^0_X(r)))
\end{equation}
$$
=\ H^1(X,\, \text{End}^P_0 (\text{Sym}^{r-1}
(\mathcal{E}_*))/\text{End}^P_{\mathcal F}(\text{Sym}^{r-1}(\mathcal{E}_*)))
$$
induced by $\Phi$ in \eqref{t8} is injective (see \eqref{t14} for the isomorphism).
The homomorphism $\mathbf h$ in \eqref{t7} gives a homomorphism of complexes
$$
{\mathbf h}\ :\ TX\otimes{\mathcal O}_X(-S)\ \longrightarrow\ \widetilde{\mathcal B}^0_\bullet
$$
because ${\mathcal D}\circ {\mathbf h}\,=\,0$ (see \eqref{e35b}), where $\mathcal D$
is the homomorphism in \eqref{t5}. Let
\begin{equation}\label{t15a}
{\mathbf h}_*\ :\ H^1(X,\, TX\otimes{\mathcal O}_X(-S))\ \longrightarrow\
{\mathbb H}^1(\widetilde{\mathcal B}^0_\bullet)
\end{equation}
be the homomorphism of hypercohomology groups induced by $\mathbf h$.

The exact sequence of complexes in \eqref{t1a} gives the following sequence of
complexes
$$
\begin{matrix}
&& 0 && 0\\
&& \Big\downarrow && \Big\downarrow\\
{\mathcal C}'_\bullet & : & {\mathcal C}'_0\, =\, \text{At}^0_X(r) &
\stackrel{\mathcal{D}}{\longrightarrow}& {\mathcal C}'_1\\
\,\,\,\, \Big\downarrow Q &&\,\,\,\,\, \,\Big\downarrow Q_0 &&\,\,\,\,\,\,\Big\downarrow Q_1\\
\widetilde{\mathcal B}^0_\bullet &:& \widetilde{\mathcal B}^0_0
& \stackrel{\mathcal{D}}{\longrightarrow}&
\widetilde{\mathcal B}^0_1\\
\,\,\,\, \Big\downarrow Q' && \Big\downarrow && \Big\downarrow\\
\widetilde{\mathcal A}'_\bullet &:& \widetilde{\mathcal A}'_0\,:=\,
\widetilde{\mathcal B}^0_0/Q_0(\text{At}^0_X(r))\,=\,
\text{End}^P_0 (\text{Sym}^{r-1}
(\mathcal{E}_*))/\text{End}^P_{\mathcal F}(\text{Sym}^{r-1}(\mathcal{E}_*))
& \longrightarrow & 0 \\
&& \Big\downarrow && \Big\downarrow\\
&& 0 && 0
\end{matrix}
$$
(see \eqref{t14}).
Since $Q'\circ Q\,=\, 0$, the homomorphism of hypercohomology groups
$$
(Q'\circ Q)_*\ :\ {\mathbb H}^1({\mathcal C}'_\bullet)\ \longrightarrow\
H^1(X,\, \text{End}^P_0 (\text{Sym}^{r-1}
(\mathcal{E}_*))/\text{End}^P_{\mathcal F}(\text{Sym}^{r-1}(\mathcal{E}_*)))
$$
vanishes. Therefore, the given condition that $\Phi_*$ in \eqref{t15} is injective implies
that
\begin{equation}\label{t16}
0\ =\ {\mathbf h}_*(H^1(X,\, TX\otimes{\mathcal O}_X(-S)))\cap
Q_*({\mathbb H}^1({\mathcal C}'_\bullet))\ \subset\ 
{\mathbb H}^1 (\widetilde{\mathcal B}^0_\bullet),
\end{equation}
where ${\mathbf h}_*$ and $Q_*$ are the homomorphisms in \eqref{t15a} and
\eqref{t9} respectively.

It was shown earlier that $Q_*$ is injective. In view of this, from \eqref{t16} it
follows immediately that the homomorphism $(P\circ Q)_*$ in \eqref{t6} is injective. Now
Corollary \ref{cor-t1} gives that the map $\mathbb{M}^0$ in \eqref{e53} is an immersion.
\end{proof}

\begin{theorem}\label{thm2}
The map $\mathbb{M}^0$ in \eqref{e53} is an immersion.
\end{theorem}

\begin{proof}
In view of Proposition \ref{prop-t1} it suffices to show that
the homomorphism of cohomologies
$$
\Phi_*\ :\ H^1(X,\, TX\otimes{\mathcal O}_X(-S))\ \longrightarrow\ 
H^1(X,\, \widetilde{\mathcal B}^0_0/Q_0({\rm At}^0_X(r)))
$$
$$
 =\ H^1(X,\,\text{End}^P_0 (\text{Sym}^{r-1}
(\mathcal{E}_*))/\text{End}^P_{\mathcal F}(\text{Sym}^{r-1}(\mathcal{E}_*)))
$$
induced by $\Phi$ in \eqref{t8} is injective (see \eqref{t14} for the isomorphism).
The homomorphism
$$
\Phi\ :\ TX\otimes{\mathcal O}_X(-S)\ \longrightarrow\ \text{End}^P_0 (\text{Sym}^{r-1}
(\mathcal{E}_*))/\text{End}^P_{\mathcal F}(\text{Sym}^{r-1}(\mathcal{E}_*))
$$
is a parabolic analog of the second fundamental form of the connection
$D$ on $\text{Sym}^{r-1}(\mathcal{E}_*)$ for the filtration in \eqref{e39}. In fact, a
parabolic ${\rm SL}(r,{\mathbb C})$--oper corresponds to an orbifold
${\rm SL}(r,{\mathbb C})$--oper; the homomorphism $\Phi$ is given by the second
fundamental from of the orbifold ${\rm SL}(r,{\mathbb C})$--oper corresponding to the
parabolic ${\rm SL}(r,{\mathbb C})$--oper $D$.

{}From the conditions on parabolic ${\rm SL}(r,{\mathbb C})$--oper it follows that
$\Phi(TX\otimes{\mathcal O}_X(-S))$ is a line subbundle of
$\text{End}^P_0 (\text{Sym}^{r-1}
(\mathcal{E}_*))/\text{End}^P_{\mathcal F}(\text{Sym}^{r-1}(\mathcal{E}_*))$, and
furthermore, the quotient
\begin{equation}\label{et2}
(\text{End}^P_0 (\text{Sym}^{r-1}
(\mathcal{E}_*))/\text{End}^P_{\mathcal F}(\text{Sym}^{r-1}(\mathcal{E}_*)))/
\Phi(TX\otimes{\mathcal O}_X(-S))
\end{equation}
has a filtration of parabolic subbundles such that each successive quotient is
a parabolic tensor power of the form ${\mathbb T}^{\otimes k}$, where $\mathbb T$
is the parabolic line bundle in \eqref{et} and $k$ is a positive integer. From this
property of the parabolic vector bundle in \eqref{et2} it follows that
$$
H^0(X,\, (\text{End}^P_0 (\text{Sym}^{r-1}
(\mathcal{E}_*))/\text{End}^P_{\mathcal F}(\text{Sym}^{r-1}(\mathcal{E}_*)))/
\Phi(TX\otimes{\mathcal O}_X(-S)))\ =\ 0.
$$
This implies that the homomorphism
$$
\Phi_*\ :\ H^1(X,\, TX\otimes{\mathcal O}_X(-S))\ \longrightarrow\ 
\ H^1(X,\,\text{End}^P_0 (\text{Sym}^{r-1}
(\mathcal{E}_*))/\text{End}^P_{\mathcal F}(\text{Sym}^{r-1}(\mathcal{E}_*)))
$$
in \eqref{t15} is injective. Now Proposition \ref{prop-t1} completes the proof.
\end{proof}

\appendix

\section{Parabolic opers}\label{APO}

The purpose of this appendix is to recall some results by K. Yokogawa \cite{Y} on Hom-sheaves, tensor products and extension classes of parabolic
bundles and to give an alternative definition of a parabolic ${\rm SL}(r)$-oper \cite{BDP} which is conceptually closer to the definition of an ordinary
${\rm SL}(r)$-oper.

\subsection{Correspondence: flags and $\mathbb{R}$-filtered sheaves} \label{corrflagfilteredsheaf}

We first recall the correspondence between a parabolic vector bundle as defined in section 2.1 and an $\mathbb{R}$-filtered sheaf 
$\{ E_t \}_{t \in \mathbb{R}}$ as introduced and studied in \cite{MY}, \cite{Y}, \cite{BY}. Using the notation of 
Section 2.1 we define for $t \,\in\, [0,\,1]$ the vector bundle $E_t$ by the following equalities 

$$
E^i_t\ = \ E\ \text {for }\ 0 \,\leq\, t \,\leq\, \alpha_{i,1}
$$
$$
E^i_t\ =\ \mathrm{ker}(E \rightarrow E_{x_i} / E_{i,j})\ \text{ for }\ \alpha_{i,j-1}
\,<\, t \,\leq\, \alpha_{i,j}
$$
$$
 E^i_t\ = \ E (-x_i)\ \text{ for }\ \alpha_{i,l_i}\, <\, t \leq 1
$$
$$
 E_t\ = \ \bigcap_{i= 1}^n E^i_t
$$

We extend to $\mathbb{R}$ by the formula $E_{t+1} \,=\, E_t(-S)$. We also denote this
$\mathbb{R}$-filtered sheaf by $E_*$. Note that $E_t \,\subset\, E_{t'}$ for any $t \,\geq\, t'$.

We recall that the family $E_*$ is left-continuous, meaning that for any $t \,\in\, \mathbb{R}$
$$ \lim_{s \to t \atop{s < t}} E_s\ =\ E_{t}. $$

The family $E_*$ is not right-continuous and we will denote for any $t \in \mathbb{R}$
$$ E_{t+} \,:=\,\lim_{s \to t \atop{s > t}} E_s.$$
Then $E_{t+}$ is a subsheaf of $E_t$ and the quotient $E_t/ E_{t+}$ is a torsion-sheaf supported at
the parabolic divisor $S$.

\subsection{Special structure and shifts}

Every vector bundle $E$ can be considered as a parabolic vector bundle $E_*$ with the special structure, defined either
by the properties $l_i\,=\, 1,\,\, \alpha_{i,1}\, =\, 0$ for $1 \,\leq\, i \,\leq\, n$, or equivalently by the equalities
$$ E_t \,:=\, E \qquad \text{for} \ t \,\in\, ]-1,\,0].$$

Given a parabolic vector bundle $E_*$ and an n-tuple $\underline{\beta}\, =\,
(\beta_1, \,\ldots , \,\beta_n )\,\in\, \mathbb{R}^n$ we define the shift $E[\underline{\beta}]_*$ by the
equalities
$$ E[\underline{\beta}]_t \,=\, \bigcap_{i=1}^n E^i_{t + \beta_i} \qquad \text{for} \ t \,\in\, \mathbb{R}.$$
Also, in order to simplify notation, we define for $\beta \in \mathbb{R}$ the shift of $E[\beta]_*$ as
$E[\beta]_*\, =\, E[\underline{\beta}]_*$ with $\beta_i \,=\, \beta$ for every $i$.

\subsection{Tensor products, Hom-sheaves and duals}

Given two parabolic vector bundles $E_*$ and $E'_*$ we define their parabolic tensor product (see e.g. 
\cite{Y} section 3) by the formula

$$ (E_* \otimes E'_*)_t := \sum_{s \in \mathbb{R}} E_s \otimes E'_{t-s} \subset E_0\otimes E'_0(*S).$$

Here $E_0\otimes E'_0(*S)$ denotes the (non-coherent) sheaf of rational sections of $E_0\otimes E'_0$ admitting 
arbitrary poles at the parabolic divisor $S$. We note that it is enough to consider the sum for $s$ running over a
subinterval of $\mathbb{R}$ of length $1$ (because of the invariance of the tensor product $E_s \otimes E'_{t-s}$ under the shift 
$s \mapsto s+1$) and that only a finite number of subsheaves $E_s \otimes E'_{t-s}$ occur.

Before defining the parabolic Hom-sheaf of two parabolic vector bundles $E_*$ and $E'_*$, we define the sheaf
$Hom(E_*, E'_*)$ as the subsheaf of $Hom(E_0,\, E'_0)$ consisting of parabolic homomorphisms, i.e., homomorphisms
$f \,:\, E_0 \,\longrightarrow\, E'_0$ satisfying 
$$f(E_t) \,\subset\, E'_t$$ 
for any $t \,\in\, [0,\,1[$ and thus for any $t \,\in\, \mathbb{R}$. Note that if $E_*$ and $E'_*$ are vector bundles, then $Hom(E_*, \,E'_*)$
is also a vector bundle. Now, we define the parabolic Hom-sheaf $Hom(E_*,\, E'_*)_*$ by the formula
$$ Hom(E_*, \,E'_*)_t \,:=\, Hom(E_*,\,E'[t]_*)$$
for any $t \,\in\, \mathbb{R}$. 

The parabolic dual $E^\vee_*$ of a parabolic vector bundle $E_*$ is by definition
$$ E^\vee_* \,:=\, Hom(E_*,\, \mathcal{O}_*)_*, $$
where $\mathcal{O}_*$ denotes the trivial bundle with the special structure.

The above definitions of parabolic tensor products, Hom-sheaves and duals extend the standard operations
on vector bundles, when considering a vector bundle as a parabolic vector bundle with its special structure. Also,
the following relations are easy to check\textcolor{red}{:}
$$ E[\underline{\beta}]_* \otimes E'[\underline{\beta'}]_* = E \otimes E'[\underline{\beta} + \underline{\beta'}]_* $$
$$ E[\underline{\beta}]_*^\vee = E^\vee[ - \underline{\beta}]_* $$
$$ E^\vee_* \otimes E'_* = Hom(E_*, E'_*)_* $$

\subsection{Cohomology of a parabolic vector bundle}

Given a parabolic vector bundle $E_*$ over the curve $X$ we define the cohomology of $E_*$ as the cohomology of the vector 
bundle $E_0$
$$ H^i(X, E_*) = H^i(X, E_0).$$

\subsection{Parabolic subbundles and parabolic degree}

We say that $E'_*$ is a parabolic subbundle of $E_*$ if there is an injective parabolic homomorphism 
$$E'_* \hookrightarrow E_*$$
with torsion-free cokernel, or equivalently, for any $t \in \mathbb{R}$, the subsheaf $E'_t$ is a subbundle of $E_t$. 

The parabolic degree of a parabolic vector bundle $E_*$ is defined as 
$$ \mathrm{\mbox{par-deg}}(E_*) \ =\ \int_0^1 \deg(E_t) dt + n \mathrm{rk}(E_*),$$
where $n$ is the number of parabolic points. We have the following formulae\textcolor{red}{:}
$$ \mathrm{\mbox{par-deg}}(E_* \otimes E'_*)\ =\
 \mathrm{rk}(E'_*)\mathrm{\mbox{par-deg}}(E_*) + \mathrm{rk}(E_*)\mathrm{\mbox{par-deg}}(E'_*),$$
$$ \mathrm{\mbox{par-deg}}(E[\underline{\beta}]_*)\ =\ \mathrm{\mbox{par-deg}}(E_*) -
\mathrm{rk}(E_*) \sum_{i=1}^n \beta_i. $$

\subsection{Canonical injections and quasi-isomorphisms}

Given a parabolic line bundle $L_*$ and a vector $\underline{\gamma} \in \mathbb{R}^n$ with $\gamma_i \geq 0$ for all
$i$, we have a canonical parabolic 
injection
$$ \iota\ :\ L_*\ \longrightarrow\ L[-\underline{\gamma}]_* $$
induced by the natural inclusions $L_{t} \,\subset\, L_{t - \gamma_i}$.

\begin{definition} \label{quasiiso}
We say that a parabolic homomorphism between two parabolic line bundles 
$$\varphi\ :\ L_*\ \longrightarrow\ M_* $$
is a quasi-isomorphism, if there exists a vector $\underline{\gamma} \in \mathbb{R}^n$ with $0 \leq \gamma_i < 1$
and a parabolic isomorphism $M_* \cong L[-\underline{\gamma}]_*$ such that via this 
isomorphism $\varphi$ identifies with the canonical injection $\iota$. In that case we say that 
$\varphi$ is a quasi-isomorphism of weight $\underline{\gamma} \in \mathbb{R}^n$.
\end{definition}

\subsection{Extensions of parabolic vector bundles}\label{extparbun}

Given two parabolic vector bundles $E_*$ and $E'_*$ we say that the parabolic vector bundle $F_*$ is an extension of
$E_*$ by $E'_*$ if there exists a short exact sequence of parabolic homomorphisms
$$ 0 \longrightarrow E'_* \longrightarrow F_* \longrightarrow E_* \longrightarrow 0. $$
By \cite{Y} Lemma 1.4 and Lemma 3.6 the isomorphism classes of extensions of $E_*$ by $E'_*$ are in one-to-one 
correspondence with the cohomology
space 
$\mathrm{Ext}^1(E_*,\, E'_*)\, =\, H^1(X,\, Hom(E_*,\, E'_*))$.

\subsection{Connections on parabolic vector bundles} Given a parabolic vector bundle $E_*$ with parabolic divisor
$S \subset X$ we define a connection $\nabla_*$ on $E_*$ as a $\mathbb{C}$-linear homomorphism
between the parabolic vector bundles $E_*$ and $E_* \otimes K[-1]_*$
\begin{eqnarray} \label{defconpar}
\nabla_*\ :\ E_*\ \longrightarrow\ E_* \otimes K[-1]_*
\end{eqnarray}
such that for every $t \in \mathbb{R}$ the map $\nabla_t: E_t \rightarrow (EK[-1])_t = E_t K(S)$ is a 
logarithmic connection with poles at the parabolic divisor $S$ and for any $t \leq t'$ we have
a commutative diagram
\[
\xymatrix{
E_{t'} \ar@{^(->}[d]\ar@{->}[r]& E_{t'}K(S) \ar@{^(->}[d]\\
E_t \ar@{->}[r]& E_{t}K(S) }
\]
where the vertical maps are the natural inclusions.

On the trivial parabolic vector bundle $\mathcal{O}_*$ there is a natural connection given by the de Rham 
differentiation and which is denoted by $d_*$
$$ d_* : \mathcal{O}_* \rightarrow K[-1]_* $$
and, fixing an integer $n$, is given for $t \in ]n-1, n]$ by differentiation of regular functions having zeros or poles of 
order $n$ at $S$
$$d_t\ :\ \mathcal{O}_t\ =\ \mathcal{O}_X(-nS)\ \longrightarrow\ K[-1]_t\ =\ K(-(n-1)S).$$

We now describe the properties of $\nabla_*$ in terms of the parabolic structure given by the flags at the parabolic 
divisor.

\begin{lemma}
Consider a connection $\nabla_*$ on $E_*$ as defined in (\ref{defconpar}). Then the logarithmic connection $\nabla_0$ on $E_0$ obtained
by putting $t\,=\,0$ satisfies 
\begin{eqnarray} \label{flagpreserving}
\mathrm{Res}(\nabla_0,\, x_i) (E_{i,j})\ \subset\ E_{i,j} 
\end{eqnarray}
for any $i = 1, \ldots, n$ and any $j = 1, \ldots , l_i$.
Conversely, any logarithmic connection $\nabla_0$ on $E_0$ satisfying (\ref{flagpreserving}) gives rise to a connection $\nabla_*$ on $E_*$.
\end{lemma}

The proof of this lemma is standard and therefore left to the reader.

\subsection{Parabolic connections on parabolic vector bundles}

Consider a connection $\nabla_*$ on a parabolic vector bundle $E_*$ as defined in (\ref{defconpar}). Then for any $t \in \mathbb{R}$
we can consider the residue at $x_i \in S$ of the logarithmic connection $\nabla_t$
$$ \mathrm{Res}(\nabla_t, x_i) \in \mathrm{End}_{\mathbb{C}}((E_{t})_{x_i}). $$
Since $\nabla_{t}$ preserves all subsheaves $E_{t'}$ for $t' \geq t$ we obtain --- by passing to the quotient 
$(E_{t}/E_{t+})_{x_i}$ --- a linear map
$$ \overline{\mathrm{Res}}(\nabla_t, x_i) \in \mathrm{End}_{\mathbb{C}}((E_{t}/E_{t+})_{x_i}),$$
and therefore an endomorphism, which we simply denote by $\overline{\mathrm{Res}}(\nabla_t)$, of the 
torsion sheaf $E_t/E_{t+}$.

With this notation, we can now define a parabolic connection on a parabolic vector bundle.

\begin{definition}
Let $E_*$ be a parabolic vector bundle. We say that a connection $\nabla_*$ on $E_*$
is a {\em parabolic} connection if for any $t \,\in\, \mathbb{R}$
$$ \overline{\mathrm{Res}}(\nabla_t)\, =\, t \mathrm{Id}.$$
\end{definition}

We leave it as an exercise to the reader to check that this definition is equivalent via the correspondence of section \ref{corrflagfilteredsheaf}
to the definition given in section 2.2, i.e., for any parabolic point $x_i \in S$ the residue of the logarithmic connection $\nabla_0$ acts as
$\alpha_{i,j} \mathrm{Id}$ on the quotient space $E_{i,j}/ E_{i,j+1}$.

We also mention the following useful facts, whose proofs are standard.

Let $(E_*,\, \nabla_*)$ and $(E'_*,\, \nabla'_*)$ be parabolic vector bundles equipped with parabolic connections. Then
\begin{itemize}
\item the parabolic tensor product $E_* \otimes E'_*$ is naturally equipped with the tensor product connection
$(\nabla \otimes \nabla')_*$, which is also parabolic.

\item the parabolic symmetric power $\mathrm{Sym}^m E_*$ is naturally equipped with the symmetric power connection
$(\mathrm{Sym}^m \nabla)_*$, which is also parabolic.

\item the de Rham differentiation $d_*$ on the trivial parabolic vector bundle $\mathcal{O}_*$ is a parabolic connection.
\end{itemize}

\begin{remark}
Note that, when considering parabolic vector bundles from the ``flag"-point of view, the relation between 
the parabolic weights of $E_*$ and those of its symmetric powers $\mathrm{Sym}^m E_*$ is quite complicated 
(as one needs to take fractional parts and reorder them in order to obtain an increasing sequence of 
parabolic weights in the interval $[0,1[$). Therefore the ``$\mathbb{R}$-filtered sheaf"-point of view is 
more adapted when considering symmetric powers, as we will need to do in the sequel.
\end{remark}

With this notation we can reformulate the following existence theorem.

\begin{theorem}[{\cite{BL}}]\label{exparconn}
The parabolic vector bundle $E_*$ admits a parabolic connection $\nabla_*$ if and only if any direct summand of $E_*$ has
parabolic degree equal to $0$.
\end{theorem}

\subsection{Parabolic ${\rm SL}(2)$-opers}\label{parsl2oper}

We now define the parabolic analogue of the Gunning bundle. 
Given the parabolic weights $\alpha_{i,1}, \alpha_{i,2}$ for $1 \leq i \leq n$ satisfying the inequalities (2.3) and the
additional assumption $\alpha_{i,1} + \alpha_{i,2} = 1$ for all $i$, we introduce the real numbers
$$ \beta_i = \alpha_{2,i} - \alpha_{1,i} \in ]0,1[ $$
and we define 
$$ \underline{\beta} = (\beta_1, \beta_2, \ldots, \beta_n ) \in \mathbb{R}^n. $$
Let $K$ be the canonical bundle of the curve $X$. Then we define the canonical parabolic vector bundle by 
$$ K^{par}_* = K[- \underline{\beta}]_*, $$
where we equip $K$ with the special structure. We also define a parabolic theta-characteristic 
$\theta^{par}_*$ as a parabolic line bundle satisfying
$$ \theta^{par}_* \otimes \theta^{par}_* = K^{par}_*.$$
One can easily check that the parabolic line bundle $\theta^{par}_*$ (respectively,
$\left(\theta^{par}_*\right)^{-1}$) corresponds via the above 
correspondence to a line bundle $M$ (respectively, $Q$) satisfying $M^2\,=\, K(-S)$ (respectively, $Q^2
\, =\, K^{-1}(-S)$) with parabolic weight
$\alpha_{i,2}$ (respectively, $\alpha_{i,1}$) at the parabolic point $x_i$ for $1\,\leq\, i \,\leq\, n$.

We define the parabolic Gunning bundle $\mathcal{G}^{par}_*$ as the unique non-split parabolic extension 
\begin{equation} \label{espargunning}
0\, \longrightarrow\, \theta^{par}_*\, \longrightarrow\, \mathcal{G}^{par}_*\, \longrightarrow\,
\left(\theta^{par}_*\right)^{-1}\, \longrightarrow\, 0.
\end{equation}
We note that the space of parabolic extensions of $\left(\theta^{par}_*\right)^{-1}$ by $\theta^{par}_*$ is 
one-dimensional, since
\begin{eqnarray*}
 \dim \mathrm{Ext}^1(\left(\theta^{par}_*\right)^{-1}, \theta^{par}_*) & = & 
\dim H^1 ( \theta^{par}_* \otimes \theta^{par}_* ) \\
& = & \dim H^1 (K[- \underline{\beta}]_*) = \dim H^1(K) = 1. 
\end{eqnarray*}
Clearly, the parabolic vector bundle $\mathcal{G}^{par}_*$ has trivial parabolic determinant since
$$\mathrm{det} \ \mathcal{G}^{par}_* \,=\, \theta^{par}_* \otimes \left(\theta^{par}_*\right)^{-1}\,
= \, \mathcal{O}_*.$$
It is easy to check that the exact sequence 
(\ref{espargunning}) equals the exact sequence (3.4) in \cite{BDP} defining the underlying parabolic vector bundle 
of a parabolic ${\rm SL}(2)$-oper.

Furthermore, any parabolic connection $\nabla_*$ on $\mathcal{G}^{par}_*$ induces a second fundamental form, which
is a $\mathcal{O}_X$-linear parabolic homomorphism
$$ \psi\ :\ \theta^{par}_*\ \longrightarrow\ (\theta^{par}_*)^{-1} \otimes K[-1]_*
\ \cong\ \theta^{par}_*[-1 + \underline{\beta}]. $$
The same argument as in the non-parabolic case shows that $\psi \,\not=\, 0$, since $\psi\,=\,0$ would imply
the existence of a parabolic 
connection on the parabolic vector bundle $\theta^{par}_*$ having parabolic degree 
$g-1 + \frac{1}{2}(\sum_{i=1}^n \beta_i)\, >\, 0$, which contradicts Theorem \ref{exparconn}. Thus $\psi$
is a quasi-isomorphism of weight $1 - \underline{\beta}$.

\subsection{Parabolic ${\rm SL}(r)$-opers}

With the above introduced notation we give a new definition of a parabolic ${\rm SL}(r)$-oper.

\begin{definition}
A parabolic ${\rm SL}(r)$-oper is a triple $(E_*, {E_{\bullet}}_*, \nabla_*)$ consisting of a
rank-$r$ parabolic vector bundle $E_*$, a filtration ${E_{\bullet}}_*$ of $E_*$ by parabolic subbundles
$$ 0 = E_{0*} \subset E_{1*} \subset E_{2*} \subset \ldots \subset E_{r-1*} \subset E_{r*} = E_* $$
with $\mathrm{rk}(E_{i*}) = i$ and a parabolic connection $\nabla_*$ on $E_*$ satisfying the following
conditions
\begin{itemize}
\item $\det(E_*, \nabla_*) = (\mathcal{O}_*, d_*)$

\item $\nabla_*(E_{i*}) \subset E_{i+1*} \otimes K[-1]_*$ for any $i = 1, \ldots , r-1$

\item There exists a vector $\underline{\beta} \in \mathbb{R}^n$ with $0 < \beta_i < 1$ such that 
for any $i = 1, \ldots, r-1$ the parabolic homomorphisms induced by $\nabla_*$ between parabolic line bundles
$$\left(E_{i*}/E_{i-1*} \right)\ \longrightarrow\ \left(E_{i+1*}/E_{i*} \right) \otimes K[-1]_* $$
are quasi-isomorphisms of weight $1 - \underline{\beta}$ (see Definition (\ref{quasiiso})).
\end{itemize}
\end{definition}

\begin{remark}
If $r\,=\,2$ one can easily show that any parabolic ${\rm SL}(2)$-oper is of the form 
$(\mathcal{G}^{par}_*, \mathcal{G}^{par}_{\bullet*}, \nabla_*)$, where $\mathcal{G}^{par}_*$ is the parabolic
Gunning bundle introduced in section \ref{parsl2oper}, $\mathcal{G}^{par}_{\bullet*}$ is given by the exact sequence 
(\ref{espargunning}) and $\nabla_*$ is any parabolic connection satisfying $\det \nabla_* = d_*$.
\end{remark}

We now show that, similar to the non-parabolic case, the underlying parabolic vector bundle of a 
parabolic ${\rm SL}(r)$-oper is a parabolic symmetric power of the parabolic Gunning bundle.

\begin{theorem}
Let $(E_*, {E_{\bullet}}_*, \nabla_*)$ be a parabolic ${\rm SL}(r)$-oper associated to the vector 
$\underline{\beta} \in \mathbb{R}^n$. Then, up to tensor product with an $r$-torsion parabolic line bundle,
we have an isomorphism between parabolic vector bundles 
$$ E_* \cong \mathrm{Sym}^{r-1} \mathcal{G}^{par}_*,$$
where $\mathcal{G}^{par}_*$ is the parabolic Gunning bundle associated to the vector 
$\underline{\beta} \,\in\, \mathbb{R}^n$.
Moreover, under this isomorphism the filtration ${E_{\bullet}}_*$ corresponds to the natural filtration
of $\mathrm{Sym}^{r-1} \mathcal{G}^{par}_*$.
\end{theorem}

\begin{proof}
In order to simplify the notation we introduce the parabolic line bundles $Q_{i*} = E_{i*}/E_{i-1*}$ for
$i=1, \ldots, r$. Then the quasi-isomorphisms $Q_{i*} \rightarrow 
Q_{i+1*} \otimes K[-1]_*$ correspond to isomorphisms
$$ Q_{i*} = Q_{i+1*} \otimes K[-\underline{\beta}]_*$$
for $i = 1, \ldots, r-1$. Iterating these formulae we can express all line bundles in terms of $Q_{r*}$
$$ Q_{r-i*} = Q_{r*} \otimes K^i[-i \underline{\beta}]. $$
Since $\det E_* = \mathcal{O}_*$, we obtain that the parabolic tensor product of all $Q_{i*}$ equals $\mathcal{O}_*$,
which leads to the isomorphism
$$Q_{r*}^r = K^{-\frac{r(r-1)}{2}}[ \frac{r(r-1)}{2} \underline{\beta}].$$
We choose a parabolic theta-characteristic $\theta^{par}_*$, i.e., a parabolic line bundle satisfying
$(\theta^{par}_*)^2 = K^{par} = K[-\underline{\beta}]$. Then the above equality is equivalent to saying that
$Q_{r*}$ and $(\theta^{par}_*)^{-(r-1)}$ differ by an $r$-torsion line bundle. So, after tensorizing $E_*$ and consequently
all quotient line bundles $Q_{i*}$ by this $r$-torsion line bundle, we can assume that 
$Q_{r*} = (\theta^{par}_*)^{-(r-1)}$. From the above formulae, we immediately obtain that
$$Q_{r-i*} = (\theta^{par}_*)^{-(r-1) + 2i}$$
for $i= 0, \ldots, r-1$. Next we will show that the natural exact sequence 
$$ 0 \longrightarrow Q_{i*} \longrightarrow E_{i+1*}/E_{i-1*} \longrightarrow Q_{i+1*} 
\longrightarrow 0$$
is the unique non-split parabolic extension of $Q_{i+1*}$ by $Q_{i*}$. By section \ref{extparbun}
parabolic extensions are parameterized by $\mathrm{Ext}^1(Q_{i+1*}, Q_{i*})$ and we have
\begin{eqnarray*}
 \dim \mathrm{Ext}^1(Q_{i+1*}, Q_{i*}) & = & 
\dim H^1 ( Q_{i+1*}^\vee \otimes Q_{i*} ) \\
& = & \dim H^1 (K[- \underline{\beta}]_*) = \dim H^1(K) = 1. 
\end{eqnarray*}
We now show that $E_{i+1*}/E_{i-1*}$ is non-split. Suppose on the contrary that we have
a direct sum decomposition 
$$E_{i+1*}/E_{i-1*} = Q_{i*} \oplus Q_{i+1*}.$$
We claim that this splitting implies that the exact sequence 
\begin{equation} \label{eseiqi}
0\ \longrightarrow\ E_{i*}\ \longrightarrow\ E_{i+1*}\ \longrightarrow\ Q_{i+1*}\ \longrightarrow\ 0
\end{equation}
also splits. To see that, we consider the long exact sequence
$$ \cdots\ \longrightarrow\ \mathrm{Ext}^1(Q_{i+1*},\, E_{i-1*})\ \longrightarrow\
\mathrm{Ext}^1(Q_{i+1*},\, E_{i*})\
\stackrel{\mu}{\longrightarrow}\ \mathrm{Ext}^1(Q_{i+1*}, Q_{i*})\ \longrightarrow\ \cdots $$
where $\mu$ is induced by the push-out under the map $E_{i*}\, \longrightarrow\, Q_{i*}$. Thus, to show that
the exact sequence (\ref{eseiqi}) splits, it will be enough to show that $\mu$ is injective. But $\mathrm{Ext}^1(Q_{i+1*}, E_{i-1*}) = 0$,
since $E_{i-1*}$ can be constructed by a series of a successive extensions of $Q_{j*}$'s for $j\leq i-1$ and we have
$$\mathrm{Ext}^1(Q_{i+1*}, Q_{j*})\ =\ 0 \qquad \text{for all} \ j \,\leq\, i-1. $$
Thus $E_{i+1*}\, =\, E_{i*} \oplus Q_{i+1*}$ and after projecting from $E_{i+1*}$ onto $E_{i*}$ the connection
$\nabla_*$ restricts to a connection on $E_{i*}$. But, for $i \,\leq\, r-1$ we have
$\mathrm{\mbox{par-deg}}(E_{i*}) \,>\, 0$, which is the
desired contradiction by Theorem \ref{exparconn}.

Finally, we invoke Theorem 4.7 \cite{JP} to conclude that, since the
rank-$2$ parabolic vector bundles $E_{i+1*}/ E_{i-1*}$ are the unique non-split parabolic extensions of 
$Q_{i+1*}$ by $Q_{i*}$ for all $i=1, \ldots,r-1$, the underlying parabolic
vector bundle $E_*$ is unique up to isomorphism. On the other hand, it is easily checked that the
parabolic symmetric power $\mathrm{Sym}^{r-1} \mathcal{G}^{par}_*$ also satisfies these properties, hence
by uniqueness
both parabolic vector bundles $E_*$ and $\mathrm{Sym}^{r-1} \mathcal{G}^{par}_*$ are isomorphic.

Note that Theorem 4.7 \cite{JP} deals with non-parabolic opers, but its extension to parabolic opers is
straightforward.
\end{proof}

\begin{remark}
The last theorem shows that the above definition of parabolic ${\rm SL}(r)$-oper coincides with \cite{BDP} Definition 5.2. 
\end{remark}

\section*{Acknowledgements}

We thank an anonymous referee for useful comments. I. Biswas is partially supported by a J. C.
Bose Fellowship (JBR/2023/000003).

\bigskip

{\bf Data availability:} Data sharing not applicable to this article as no datasets were generated or analysed during the current study.

\bigskip

{\bf Conflict of interest:} 
The authors have no conflict of interest to declare that are relevant to this article.


\end{document}